\documentclass[dvips,preprint]{imsart}
\RequirePackage[OT1]{fontenc}
\usepackage{amsthm,amsmath}
\RequirePackage[numbers]{natbib}
\RequirePackage[colorlinks,citecolor=blue,urlcolor=blue]{hyperref}
\RequirePackage{hypernat}
\usepackage{enumerate}
\usepackage{amsfonts,amssymb}
\usepackage{graphicx,graphics}
\usepackage{multirow,multicol}
\startlocaldefs
\DeclareMathOperator*{\card}{Card}
\DeclareMathOperator*{\supp}{supp}
\DeclareMathOperator*{\Lper}{{\mathbb{L}^2_{\mathrm{per}}}}
\DeclareMathOperator*{\FT}{\mathcal{F}}
\newcommand{\Bspr}{{\mathbf{B}}^s_{p,r}(M)}
\newcommand{\ophi}{\overline{\phi}}
\newcommand{\opsi}{\overline{\psi}}
\theoremstyle{plain}
\newtheorem{theorem}{Theorem}[section]
\newtheorem{proposition}{Proposition}[section]
\newtheorem{remark}{Remark}[section]
\newtheorem{lemma}{Lemma}[section]
\newtheorem{example}{Example}[section]
\endlocaldefs
\makeindex
\begin{document}
\begin{frontmatter}

\title{Block thresholding for wavelet-based estimation of function derivatives from a heteroscedastic multichannel convolution model}
\runtitle{Wavelet multichannel deconvolution}

\begin{aug}
\author{\fnms{Fabien} \snm{Navarro}\ead[label=e3]{fabien.navarro@greyc.ensicaen.fr}}
\address{Laboratoire de Math\'{e}matiques Nicolas Oresme,\\
CNRS-UMR 6139, Universit\'{e} de Caen, Campus II, 14032 Caen, France\\
GREYC CNRS-ENSICAEN-Universit\'{e} de Caen, 14050 Caen France\\
\printead{e3}}
\author{\fnms{Christophe} \snm{Chesneau}\corref{}\ead[label=e1]{chesneau@math.unicaen.fr}}
\address{Laboratoire de Math\'{e}matiques Nicolas Oresme,\\
CNRS-UMR 6139, Universit\'{e} de Caen, Campus II, BP 5186\\
Boulevard du Mar\'{e}chal Juin, F-14032 Caen, France\\
\printead{e1}}
\author{\fnms{Jalal} \snm{Fadili}\ead[label=e2]{Jalal.Fadili@greyc.ensicaen.fr}}
\address{GREYC CNRS-ENSICAEN-Universit\'{e} de Caen,\\
6 Bd du Mar\'echal Juin, 14050 Caen France\\
\printead{e2}}
\medskip
\and
\author{\fnms{Taoufik} \snm{Sassi}\ead[label=e4]{taoufik.sassi@math.unicaen.fr}}
\address{Laboratoire de Math\'{e}matiques Nicolas Oresme, \\
CNRS-UMR 6139, Universit\'{e} de Caen, Campus II, BP 5186 \\
Boulevard du Mar\'{e}chal Juin, F-14032 Caen, France \\
\printead{e4}}
\runauthor{F. Navarro et al.}
\end{aug}

\begin{abstract}
We observe $n$ heteroscedastic stochastic processes $\{Y_v(t)\}_{v}$, where for any $v\in\{1,\ldots,n\}$ and $t \in [0,1]$, $Y_v(t)$ is the convolution product of an unknown function $f$ and a known blurring function $g_v$ corrupted by Gaussian noise. Under an ordinary smoothness assumption on $g_1,\ldots,g_n$, our goal is to estimate the $d$-th derivatives (in weak sense) of $f$ from the observations. We propose an adaptive estimator based on wavelet block thresholding, namely the "BlockJS estimator". Taking the mean integrated squared error (MISE), our main theoretical result investigates the minimax rates over Besov smoothness spaces, and shows that our block estimator can achieve the optimal minimax rate, or is at least nearly-minimax in the least favorable situation. We also report a comprehensive suite of numerical simulations to support our theoretical findings. The practical performance of our block estimator compares very favorably to existing methods of the literature on a large set of test functions.
\end{abstract}

\begin{keyword}[class=AMS]
\kwd[Primary ]{ 62G07}
\kwd{62G20}
\kwd[; secondary ]{62F12}
\end{keyword}

\begin{keyword}
\kwd{deconvolution}
\kwd{multichannel observations}
\kwd{derivative estimation}
\kwd{wavelets}
\kwd{block thresholding}
\kwd{minimax}
\end{keyword}
\tableofcontents
\end{frontmatter}
%
%
\section{Introduction}
\subsection{Problem statement}
Suppose that we observe $n$ stochastic processes $Y_1(t),\ldots,Y_n(t)$, $t\in \lbrack 0,1 \rbrack$ where, for any $v\in \{1,\ldots,n\}$,
\begin{equation}\label{sous}
dY_v(t)= (f\star g_v)(t)dt + \epsilon dW_v(t), \qquad t \in [0,1], \qquad n \in \mathbb{N}^*,
\end{equation}
$\epsilon>0$ is the noise level, $(f\star g_v)(t)=\int_{0}^{1}f(t-u)g_v(u)du$ denotes the convolution product on $[0,1]$, $W_1(t),\ldots,W_n(t)$ are $n$ unobserved independent standard Brownian motions, for any $v\in \{1,\ldots,n\}$, $g_{v}  : [0,1]\rightarrow \mathbb{R}$ is a known blurring function and $f : [0,1]\rightarrow \mathbb{R}$ is the unknown  function that we target.
We assume that $f$ and $g_1,\ldots,g_n$ belong to $\Lper([0,1])=\{ h ;$ \ $h$ is $1$-periodic on $[0,1]$ and $\int_{0}^{1}h^2(t)dt<\infty\}$. The focus of this paper is to estimate $f$ and its derivatives $f^{(d)}$ (to be understood in weak or distributional sense) from $Y_1(t),\ldots,Y_n(t)$, $t\in \lbrack 0,1 \rbrack$. This is in general a severely ill-posed inverse problem. Application fields cover biomedical imaging, astronomical imaging, remote-sensing, seismology, etc. This list is by no means exhaustive.
\\

In the following, any function $h \in \Lper(\lbrack 0,1 \rbrack)$ can be represented by its Fourier series
\[
h(t) = \sum_{\ell \in \mathbb{Z}} \FT(h)(\ell) e^{2 i\pi \ell t}, \ \ \ \ \ \ t\in [0,1],
\]
where the equality is intended in mean-square convergence sense, and $\FT_{\ell}(h)$ denotes the Fourier series coefficient given by 
\[
\FT(h)(\ell)=\int_{0}^{1}h(t)e^{-2 i\pi \ell t}dt, \qquad \ell\in \mathbb{Z},
\]
whenever this integral exists. The notation $\overline{~\cdot~}$ will stand for the complex conjugate.

\subsection{Overview of previous work}
There is an extensive statistical literature on wavelet-based deconvolution problems. For obvious space limitations, we only focus on some of them.

In the special case where $g_1=\dots=g_n$, \eqref{sous} reduces to the form
\begin{equation}\label{sous2}
d\widetilde{Y}(t)= (f \star g_1)(t)dt + \epsilon n^{-1/2}d\widetilde{W}(t),\qquad t\in [0,1],
\end{equation}
where $\widetilde{Y}(t)=(1/n)\sum_{v=1}^nY_v(t)$, and $\widetilde{W}(t)=(1/n^{1/2})\sum_{v=1}^nW_v(t)$ is standard Brownian motion. In such a case, \eqref{sous2} becomes the standard deconvolution which attracted attention of a number of researchers spanning a wide range of fields including signal processing and statistics. For instance, wavelet-based estimators of $f$ have been constructed and their asymptotic performance investigated in a number of papers, see e.g. \cite{Donoho95,cavalier08,cavalier6,cavalier02,chesneau13,raimondo2}. When $g_1,\ldots,g_n$ are not necessarily equal, estimators of $f$ and their minimax rates under the mean integrated squared error (MISE) over Besov balls were proposed in \cite{candi,penskyy,penskyyy,penskyyyy}. These authors develop  wavelet thresholding estimators (hard thresholding in \cite{candi,penskyyyy} and block thresholding in \cite{penskyy,penskyyy}) under various assumptions on $g_1,\ldots,g_n$ (typically, ordinary smooth and super-smooth case, or boxcar blurring functions).

Estimating the derivative of a function on the basis of noisy and blurred observations is of paramount importance in many fields such as signal processing, control or mathematical physics. For instance detecting the singularities of $f$ or characterizing its concavity or convexity properties is a longstanding problem in signal processing. The estimation of the derivatives from noisy solutions of ordinary or partial differential equations is typical in many areas of mathematical physics such as astronomy or fluid dynamics. In the case where $d=0$, several examples of recovering initial/boundary conditions from observations of a noisy and blurred solution of a PDE (\textit{e.g.}, parabolic, hyperbolic or elliptic PDE) are given in \cite{penskyy}. For higher-order derivatives (typically $d=1$ or $d=2$), there are also physical applications where such a model occurs. We mention for instance frequency self-deconvolution encountered in the field of chemometrics. In this context, the first and the second derivatives can be used to detect important information or features in the raw spectra; see e.g. \cite{Mbaideen11}). The wavelet estimator developed in the paper could be an interesting alternative to commonly used methods in this area.

The derivatives estimation have already been investigated from several standard nonparametric models. If we only restrict the review to wavelet methods, we refer to \cite{cai02,chesneauder} for model \eqref{sous2} and to \cite{prakasa2,chaubey1,chaubey2} for density estimation problems.

\subsection{Contributions and relation to prior work}
In this paper, considering an appropriate ordinary smoothness assumption on $g_1,\ldots,g_n$, we develop an adaptive wavelet-based block estimator $\widehat{f^{(d)}}$ of $f^{(d)}$ from \eqref{sous}, $d\in \mathbb{N}$. It is constructed using a periodized Meyer wavelet basis and a particular block thresholding rule which goes by the the name of BlockJS; see \cite{cai} for the original construction of BlockJS in the standard Gaussian noise model, and \cite{cavalier01,cai02,tsybakov,cfs} for further developments on BlockJS. Adopting the minimax approach under the MISE over Besov balls, we investigate the upper bounds of our estimator. This is featured in Theorem~\ref{maison}. We prove that the rates of our estimator are nearly optimal by establishing a lower bound as stated in Theorem~\ref{maison2}.

Our work is related to some prior art in the literature. To the best of our knowledge, the closest ones are those of \cite{penskyy,penskyyy}. For the case where $d=0$ and the blurring function is ordinary-smooth or super-smooth, \cite[Theorems 1 and 2]{penskyy,penskyyy} provide the upper and lower bounds of the MISE over Besov balls for a block hard thresholding estimator from the functional deconvolution model\footnote{This is a more general model which reduces to the multichannel deconvolution model when observed at a finite number of distinct points, see \cite[Section~5]{penskyy} for further details.}. These bounds match ours but only for $d=0$. In this respect, our work goes one step further as it tackles the estimation (with a different wavelet estimator) of $f$ and its derivatives. As far as the methods of proof are concerned, we use similar tools (concentration and moment inequalities as well as the general result in \cite{cfs}) as theirs for the upper bound, but the proof of the lower bounds are different. However unlike \cite{penskyy}, we only cover the ordinary smooth convolution, while their results apply also to the super-smooth case. On the practical side, for $d=0$, we will show in Section~\ref{sec:sim} that BlockJS behaves better than block hard thresholding \cite[(2.9)]{penskyyy} over several test functions that contain different degrees of irregularity.

\subsection{Paper organization}
The paper is organized as follows. Section~\ref{besov} gives a brief account of periodized Meyer wavelets and Besov balls. Section~\ref{sec:estim} states ordinary smoothness assumption on $g_1,\ldots,g_n$, and then constructs the BlockJS-based estimator. The minimax upper and lower bounds of this estimator are investigated in Section~\ref{sec:minimax}. Section~\ref{sec:sim} describes and discusses the simulation results, before drawing some conclusions in Section~\ref{sec:conclusion}. The proofs are deferred to Section~\ref{sec:proofs} awaiting inspection by the interested reader.

\section{Wavelets and Besov balls}\label{besov}
\subsection{Periodized Meyer Wavelets}\label{perio}
We consider an orthonormal wavelet basis generated by dilations and translations of  a "father" Meyer-type wavelet $\phi$ and a "mother" Meyer-type wavelet $\psi$. These wavelets enjoy the following features.
\begin{enumerate}[$\bullet$]
\item They are smooth and frequency band-limited, i.e. the Fourier transforms of $\phi$ and $\psi$ have compact supports with
\begin{equation}\label{grop}
\begin{cases} 
\supp \left( \FT(\phi)\right) \subset \lbrack -4\pi 3^{-1}, 4\pi 3^{-1}\rbrack,\\
\supp\left( \FT(\psi)\right) \subset \lbrack -8\pi 3^{-1}, -2\pi 3^{-1}\rbrack \cup \lbrack 2\pi 3^{-1}, 8\pi 3^{-1}\rbrack,
\end{cases}
\end{equation}
where $\supp$ denotes the support.
\item The functions $(\phi,\psi)$ are $C^\infty$ as their Fourier transforms have a compact support, and $\psi$ has an infinite number of vanishing moments as its Fourier transform vanishes in a neighborhood of the origin:
\begin{equation}\label{reg2}
\int_{-\infty}^{+\infty}t^u\psi(t)dt=0, \quad \forall ~ u \in \mathbb{N}.
\end{equation}
\item If the Fourier transforms of $\phi$ and $\psi$ are also in $C^m$ for a chosen $m \in \mathbb{N}$, then it can be easily shown that $\phi$ and $\psi$ obey
\begin{equation}\label{ratratrat}
|\phi(t)|=O\left((1+|t|)^{-m-1}\right), \qquad |\psi(t)|=O\left((1+|t|)^{-m-1}\right)
\end{equation}
for every $t \in \mathbb{R}$.
\end{enumerate}

For the purpose of this paper, we use the periodized wavelet bases on the unit interval. For any $t \in[0,1]$, any integer $j$ and any $k\in\{0,\ldots,2^j-1\}$, let 
\begin{equation*}
\phi_{j,k}(t)=2^{j/2}\phi (2^j t-k), \qquad \psi_{j,k}(t)=2^{j/2}\psi(2^j t-k)
\end{equation*}
 be the elements of the wavelet basis, and
\begin{equation*}
\phi^{\mathrm{per}}_{j,k}(t)=\sum_{\ell\in\mathbb{Z}}\phi_{j,k}(t-\ell), \qquad
\psi^{\mathrm{per}}_{j,k}(t)=\sum_{\ell\in\mathbb{Z}}\psi_{j,k}(t-\ell),
\end{equation*}
 their periodized versions. There exists an integer $j_*$ such that the collection $\left\lbrace \phi^{\mathrm{per}}_{j_*,k}(\cdot), k \in \{0,\ldots,2^{j_*}-1\}; \ \psi^{\mathrm{per}}_{j,k}(.), \quad j \ge j_*, \ k\in \{0,\ldots,2^j-1\}\right\rbrace $ forms an orthonormal basis of $\Lper(\lbrack0,1 \rbrack)$. In what follows, the superscript ``per'' will be dropped from $\phi^{\mathrm{per}}$ and $\psi^{\mathrm{per}}$ to lighten the notation.
\\

Let $l \ge j_*$, any function $h \in \Lper(\lbrack 0,1 \rbrack)$ can be expanded into a wavelet series as
\[
h(t)= \sum_{k=0}^{2^l-1}\alpha_{l,k}\phi_{l,k}(t)  +\sum_{j=l}^{\infty}  \sum_{k=0}^{2^j-1} \beta_{j,k}\psi_{j,k}(t), \qquad t \in [0,1],
\]
where
\begin{equation}\label{coef}
\alpha_{l,k}=\int_{0}^{1}h(t)\ophi_{l,k}(t)dt, \qquad \beta_{j,k}=\int_{0}^{1}h(t)\opsi_{j,k}(t)dt.
\end{equation}
See \cite[Vol. 1 Chapter III.11]{meyer} for a detailed account on periodized orthonormal wavelet bases.

\subsection{Besov balls}
Let $0 < M < \infty$, $s>0$, $1 \leq p, r \le \infty$. Among the several characterizations of Besov spaces for periodic functions on $\mathbb{L}^p([0,1])$, we will focus on the usual one based on the corresponding coefficients in a sufficiently $q$-regular (periodized) wavelet basis ($q=\infty$ for Meyer wavelets). More precisely, we say that a function $h$ belongs to the Besov ball $\Bspr$ if and only if $\int_{0}^1 |h(t)|^p dt \leq M$, and there exists a constant $M^*>0$ (depending on $M$) such that the associated wavelet coefficients \eqref{coef} satisfy
\begin{align}\label{besovv}
 2^{j_*(1/2-1/p )}\left(\sum_{k=0}^{2^{j_*}-1}|\alpha_{j_*,k}|^{p}\right)^{1/p}& + \left(\sum_{j=j_*}^{\infty} \left(2^{j(s+1/2-1/p )}\left(\sum_{k=0}^{2^{j}-1}|\beta_{j,k}|^{p}\right)^{1/p}\right)^{r}\right)^{1/r}\nonumber\\
 &\le  M^*,
\end{align}
with a smoothness parameter $0 < s < q$, and the norm parameters $p$ and $r$. Besov spaces capture a variety of smoothness features in a function including spatially inhomogeneous behavior, see \cite{meyer}.

\section{The deconvolution BlockJS estimator}
\label{sec:estim}
\subsection{The ordinary smoothness assumption}
In this study, we focus on the following particular ordinary smoothness assumption on $g_1,\ldots,g_n$. We assume that there exist three constants, $c_g>0$, $C_g>0$ and $\delta> 1$, and $n$ positive real numbers $\sigma_1, \ldots,\sigma_n$ such that, for any $\ell \in \mathbb{Z}$ and any $v\in \{1,\ldots,n\}$, 
\begin{equation}\label{cond}
c_g\frac{1}{(1+ \sigma_v^{2}\ell^2)^{\delta/2}}\le |\FT(g_v)(\ell)|\le C_g\frac{1}{(1+ \sigma_v^{2}\ell^2)^{\delta/2}}.
\end{equation}
This assumption controls the decay of the Fourier coefficients of $g_1,\ldots,g_n$, and thus the smoothness of $g_1,\ldots,g_n$. It is a standard hypothesis usually adopted in the field of nonparametric estimation for deconvolution problems. See e.g. \cite{pevi}, \cite{fan} and \cite{raimondo2}. 

\begin{example}
\label{ex:lap}
Let $\tau_1,\ldots,\tau_n$ be $n$ positive real numbers. For any $v\in \{1,\ldots,n\}$, consider the square-integrable $1$-periodic function $g_v$ defined by
\[
g_v(t)=\frac{1}{\tau_v}\sum_{m\in \mathbb{Z}}e^{-|t+m|/\tau_v}, \qquad t \in [0,1].
\] 
Then, for any $\ell \in \mathbb{Z}$,  $\FT(g_v)(\ell)=2\left( 1+4\pi^2 \ell^2 \tau_v^2\right)^{-1}$ and \eqref{cond} is satisfied with $\delta=2$ and $\sigma_v=2\pi \tau_v$.
\end{example}

In the sequel, we set
\begin{equation}\label{mop}
\rho_n=\sum_{v=1}^n \frac{1}{(1+\sigma_v^{2})^{\delta}}.
\end{equation}
For a technical reason that is not restrictive at all (see Section~\ref{sec:proofs}), we suppose that $\rho_n\ge e$ and $\lim_{n\rightarrow \infty} (\ln \rho_n)^{v}\rho_n^{-1}=0$ for any $v>0$. 

\subsection{BlockJS estimator}
\label{subsec:blockJSestim}
We suppose that $f^{(d)}\in \Lper([0,1])$ and that the ordinary smoothness assumption \eqref{cond} holds, where $\delta$ refers to the exponent in the assumption. We are ready to construct our adaptive procedure for the estimation of $f^{(d)}$. 

Let $j_1 = \lfloor \log_2 (\log \rho_n) \rfloor$ be the coarsest resolution level, and $j_2=\lfloor (1/(2\delta+2d+1))\log_2 (\rho_n/\log \rho_n) \rfloor$, where, for any $a\in \mathbb{R}$, $\lfloor a \rfloor$ denotes the integer part of $a$. For any $j\in\{j_1,\ldots,j_2\}$, let $L= \lfloor \log \rho_n \rfloor$ be the block size.

Let $A_j=\lbrace 1,\ldots, \lfloor 2^{j}L^{-1} \rfloor\rbrace$ be the set indexing the blocks at resolution $j$. For each $j$, let $\{B_{j,K}\}_{K \in A_j}$ be a uniform and disjoint open covering of $\{0,\ldots,2^{j}-1\}$, i.e. $\bigcup_{K\in {A}_j}B _{j,K} =\{0,\ldots,2^j-1\}$ and for any $(K,K') \in {A}_j^2$ with $K \not = K'$, $B_{j,K}\cap B_{j,K'}=\varnothing$ and $\card(B_{j,K})=L$, where $B_{j,K} = \lbrace k\in \{0,\ldots,2^{j}-1\}; \, (K-1)L\le k\le K L-1 \rbrace$ is the $K$th block.
\\

We define the Block James-Stein estimator (BlockJS) of $f^{(d)}$ by
\begin{equation}\label{lecue}
\widehat{f^{(d)}}(t)=\sum_{k=0}^{2^{j_1}-1}\widehat \alpha_{j_1,k}\phi_{j_1,k}(t)+
\sum_{j=j_1}^{j_2} \sum_{K\in A_j }\sum_{k\in B_{j,K}} \widehat \beta_{j,k}^* \psi_{j,k}(t), \quad t \in \lbrack 0,1 \rbrack ,
\end{equation}
where for any resolution $j$ and position $k \in B_{j,K}$ within the $K$th block, the wavelet coefficients of $f^{(d)}$ are estimated via the rule
\[
\widehat \beta_{j,k}^*=\widehat \beta_{j,k} \left(1- \displaystyle\frac{\lambda \epsilon^2 {\rho_n}^{-1}2^{2j (\delta+d) }}{\tfrac{1}{L}\sum_{k\in B_{j,K}}|\widehat \beta_{j,k}|^2}\right)_{+},
\]
with, for any $a\in \mathbb{R}$, $(a)_{+}=\max(a,0)$, $\lambda>0$, and $\widehat \alpha_{j_1,k}$ and $\widehat \beta_{j,k}$ are respectively the empirical scaling and detail coefficients, defined as
\begin{equation*}
\widehat \alpha_{j_1,k} =\frac{1}{\rho_n}\sum_{v=1}^n \frac{1}{(1+\sigma_v^2)^{\delta}}\sum_{\ell\in\mathcal{D}_{j_1}}(2\pi i\ell)^d\frac{\overline{\FT\left(\phi_{j_1,k}\right)}(\ell)}{\FT(g_v)(\ell)}\int_{0}^{1}e^{- 2\pi i\ell t}dY_v(t)
\end{equation*}
and
\begin{equation*}
 \widehat \beta_{j,k}=\frac{1}{\rho_n}\sum_{v=1}^n \frac{1}{(1+\sigma_v^2)^{\delta}}\sum_{\ell\in\mathcal{C}_j}(2\pi i\ell)^d\frac{\overline{\FT\left(\psi_{j,k}\right)}(\ell)}{\FT(g_v)(\ell)}\int_{0}^{1}e^{- 2\pi i\ell t}dY_v(t).
\end{equation*}
Notice that thanks to \eqref{grop}, for any $j\in \{j_1,\ldots,j_2\}$ and $k \in \{0,\ldots,2^{j}-1\}$
\begin{equation}\label{cal}
\left\{
\begin{aligned}
 \mathcal{D}_{j_1}&=\supp\left(\FT\left(\phi_{j_1,k}\right)\right) \subset \lbrack -4\pi 3^{-1}2^{j_1}, 4\pi 3^{-1}2^{j_1}\rbrack,\\
 \mathcal{C}_{j}&=\supp\left(\FT\left(\psi_{j,k}\right)\right)\subset \lbrack -8\pi 3^{-1}2^{j}, -2\pi 3^{-1}2^{j}\rbrack \cup \lbrack 2\pi 3^{-1}2^{j}, 8\pi 3^{-1}2^{j}\rbrack.
\end{aligned}
\right.
\end{equation}

\section{Minimaxity results of BlockJS over Besov balls}
\label{sec:minimax}
\subsection{Minimax upper-bound for the MISE}
Theorem \ref{maison} below determines the rates of convergence achieved by $\widehat{f^{(d)}}$ under the MISE over  Besov balls.
\begin{theorem}\label{maison}
Consider the model \eqref{sous} and recall that we want to estimate $f^{(d)}$ with $d\in \mathbb{N}$. Assume that  $(\phi,\psi)$ satisfy \eqref{ratratrat} for some $m \ge d$ and \eqref{cond} is satisfied.
Let $\widehat{f^{(d)}}$ be the estimator defined by \eqref{lecue} with a large enough $\lambda$. Then there exists a constant $C>0$ such that,
for any $M>0$, $p\ge 1$, $r\ge 1$, $ s>1/p$ and $n$ large enough, we have
\[
\sup_{f^{(d)}\in \Bspr}\mathbb{E}\left(\int_{0}^{1}\left(\widehat{f^{(d)}}(t)-f^{(d)}(t)\right)^2dt\right)\le C \varphi_n,
\]
where
\begin{equation*}
\varphi_n =\left\{
\begin{aligned}
&\rho_n^{-2s/(2s+2\delta+2d+1)},&\text{if}&~p\ge 2,\\
&(\log \rho_n/\rho_n)^{2s/(2s+2\delta+2d+1)},&\text{if}&~p\in[1,2),s>(1/p-1/2)(2\delta+2d+1).
\end{aligned}
\right.
\end{equation*}
\end{theorem}

Theorem \ref{maison} will be proved using the more general theorem \cite[Theorem 3.1]{cfs}. 
To apply this result, two conditions on the wavelet coefficients estimator are required: a moment condition and a concentration condition. They are established in Propositions \ref{copp} and \ref{cop}, see Section~\ref{sec:proofs} .

\begin{remark}
Theorem~\ref{maison} can be generalized by allowing the decay exponent $\delta$ of $g_1,\ldots,g_n$ to vary across the channels. More precisely, our ordinary assumption \eqref{cond} is uniform over the channels, in the sense that we assumed that the decay exponent $\delta$ is the same for all $v$. Assume now that there exist three constants, $c_g>0$, $C_g>0$, and $2n$ positive real numbers $\sigma_1, \ldots,\sigma_n,\delta_1,\ldots,\delta_n$ with $\min (\delta_1,\ldots,\delta_n)>1/2$ such that, for any $\ell \in \mathbb{Z}$ and any $v\in \{1,\ldots,n\}$,
\begin{equation*}
|\FT(g_v)(\ell)|\le C_g\frac{1}{(1+\sigma_v^2\ell^2)^{\delta_v/2}}.
\end{equation*}
Set 
\[
\omega=\min(\delta_1,\ldots,\delta_n).
\]
Let us now define the BlockJS estimator of $f^{(d)}$ as in Section~\ref{subsec:blockJSestim} but with the weights
\[
 \rho_n=\sum_{v=1}^n \frac{1}{(1+\sigma_v^{2})^{\delta_v}},
\]
resolution levels $j_1 = \lfloor \log_2 (\log \rho_n) \rfloor$, $j_2=\lfloor (1/(2\omega+2d+1))\log_2 (\rho_n/\log \rho_n) \rfloor$, and
\[
\widehat \beta_{j,k}^* = \widehat \beta_{j,k} \left(1- \displaystyle\frac{\lambda \epsilon^2 {\rho_n}^{-1}2^{2j (\omega+d) }}{\tfrac{1}{L}\sum_{k\in B_{j,K}}|\widehat \beta_{j,k}|^2}\right)_{+},
\]
\begin{equation*}
\widehat \alpha_{j_1,k} =\frac{1}{\rho_n}\sum_{v=1}^n \frac{1}{(1+\sigma_v^2)^{\delta_v}}\sum_{\ell\in\mathcal{D}_{j_1}}(2\pi i\ell)^d\frac{\overline{\FT\left(\phi_{j_1,k}\right)}(\ell)}{\FT(g_v)(\ell)}\int_{0}^{1}e^{- 2\pi i\ell t}dY_v(t)
\end{equation*}
and
\begin{equation*}
\widehat \beta_{j,k}=\frac{1}{\rho_n}\sum_{v=1}^n \frac{1}{(1+\sigma_v^2)^{\delta_v}}\sum_{\ell\in\mathcal{C}_j}(2\pi i\ell)^d\frac{\overline{\FT\left(\psi_{j,k}\right)}(\ell)}{\FT(g_v)(\ell)}\int_{0}^{1}e^{- 2\pi i\ell t}dY_v(t).
\end{equation*}
Then Theorem~\ref{maison} holds with the convergence rate
\begin{equation*}
\varphi_n =\left\{
\begin{aligned}
&\rho_n^{-2s/(2s+2\omega+2d+1)},&\text{if}&~p\ge 2,\\
&(\log \rho_n/\rho_n)^{2s/(2s+2\omega+2d+1)},&\text{if}&~p\in[1,2),s>(1/p-1/2)(2\omega+2d+1).
\end{aligned}
\right.
\end{equation*}
Of course, this generalization encompasses the statement of Theorem~\ref{maison}. This generalized result also tells us that the adaptivity of the estimator makes its performance mainly influenced by the best $g_i$, that is the one with the smallest $\delta_i$, which is a nice feature.
\end{remark}

\subsection{Minimax lower-bound for the MISE}
We now turn to the lower bound of the MISE to formally answer the question whether $\varphi_n$ is indeed the optimal rate of convergence or not. This is the goal of Theorem~\ref{maison2} which gives a positive answer. 

\begin{theorem}\label{maison2}
Consider the model \eqref{sous} and recall that we want to estimate $f^{(d)}$ with $d\in \mathbb{N}$. Assume that \eqref{cond} is satisfied. Then there exists a constant $c>0$ such that,
for any $M>0$, $p\ge 1$, $r\ge 1$, $ s>1/p$ and $n$ large enough, we have
\begin{equation*}
\inf_{\widetilde{f^{(d)}}}\sup_{f^{(d)}\in \Bspr}\mathbb{E}\left(\int_{0}^{1}\left(\widetilde{f^{(d)}}(t)-f^{(d)}(t)\right)^2dt\right)\ge c \varphi^*_n,
\end{equation*}
where
\begin{equation*}
\varphi^*_n = (\rho^*_n)^{-2s/(2s+2\delta+2d+1)}, \qquad \rho^*_n=\sum_{v=1}^n\sigma_v^{-2\delta}.
\end{equation*}
and the infimum is taken over all the estimators $\widetilde{f^{(d)}}$ of $f^{(d)}$.
\end{theorem}

It can then be concluded from Theorem~\ref{maison} and Theorem~\ref{maison2} that the rate of convergence $\varphi_n$ achieved by $\widehat{f^{(d)}}$ is near optimal in the minimax sense. Near minimaxity is only due to the case $p \in [1,2)$ and $s>(1/p-1/2)(2\delta+2d+1)$ where there is an extra logarithmic term. 

\section{Simulations results}
\label{sec:sim}
In the following simulation study we consider the problem of estimating one of the derivatives of a function $f$ from the heteroscedastic multichannel deconvolution model \eqref{sous}. Three test functions (``Wave'', ``Parabolas'' and ``TimeShiftedSine'', initially introduced in \citep{marron}) representing different degrees of smoothness were used (see \textsc{Fig}~\ref{fig:orig}). The ``Wave'' function was used to illustrate the performance of our estimator on a smooth function. Note that the ``Parabolas'' function has big jumps in its second derivative.

\begin{figure}[t]
\hspace*{-0.5cm}
\begin{tabular}{ccc}
\begin{minipage}{0.32\textwidth}
\includegraphics[width=\textwidth]{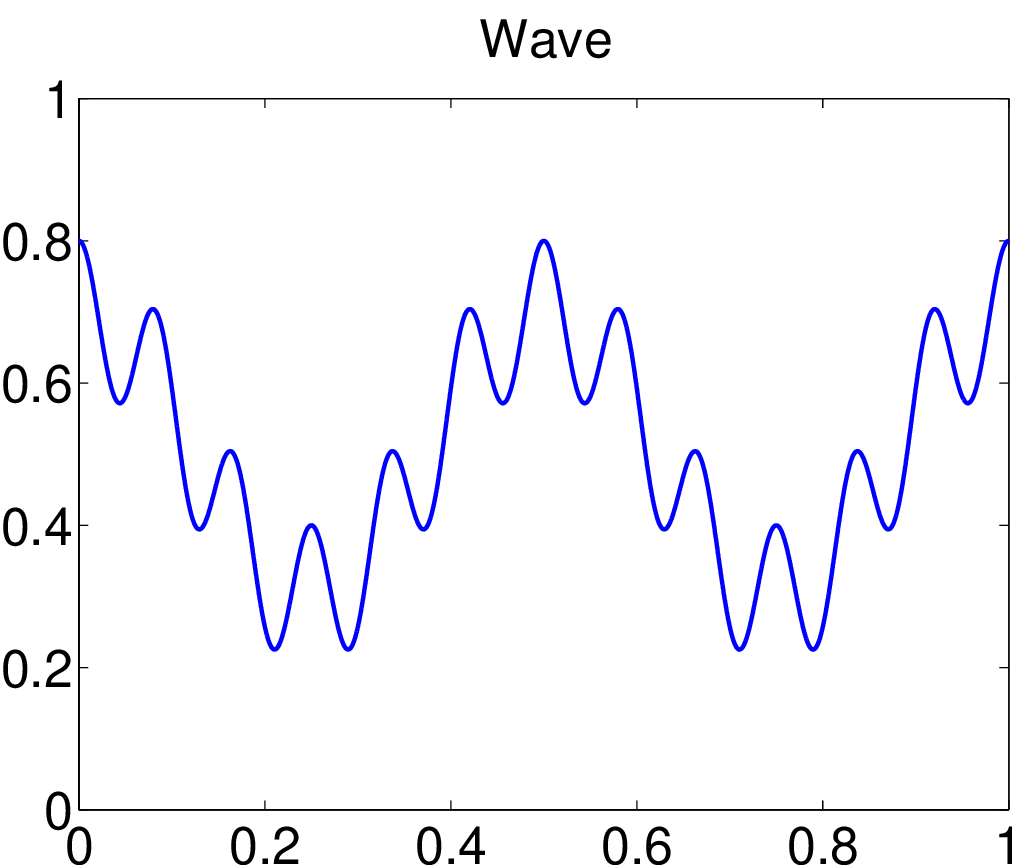} \\
\centerline{(a)}
\end{minipage}&
\begin{minipage}{0.32\textwidth}
\includegraphics[width=\textwidth]{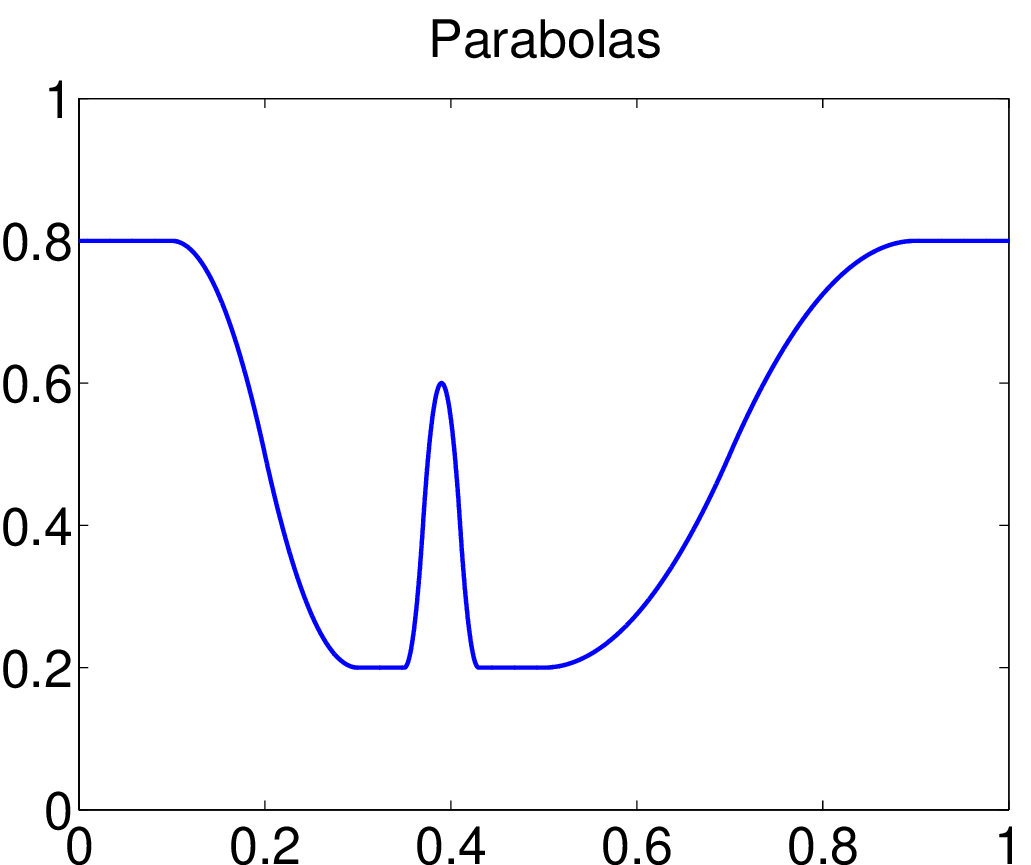} \\
\centerline{(b)}
\end{minipage}&
\begin{minipage}{0.32\textwidth}
\includegraphics[width=\textwidth]{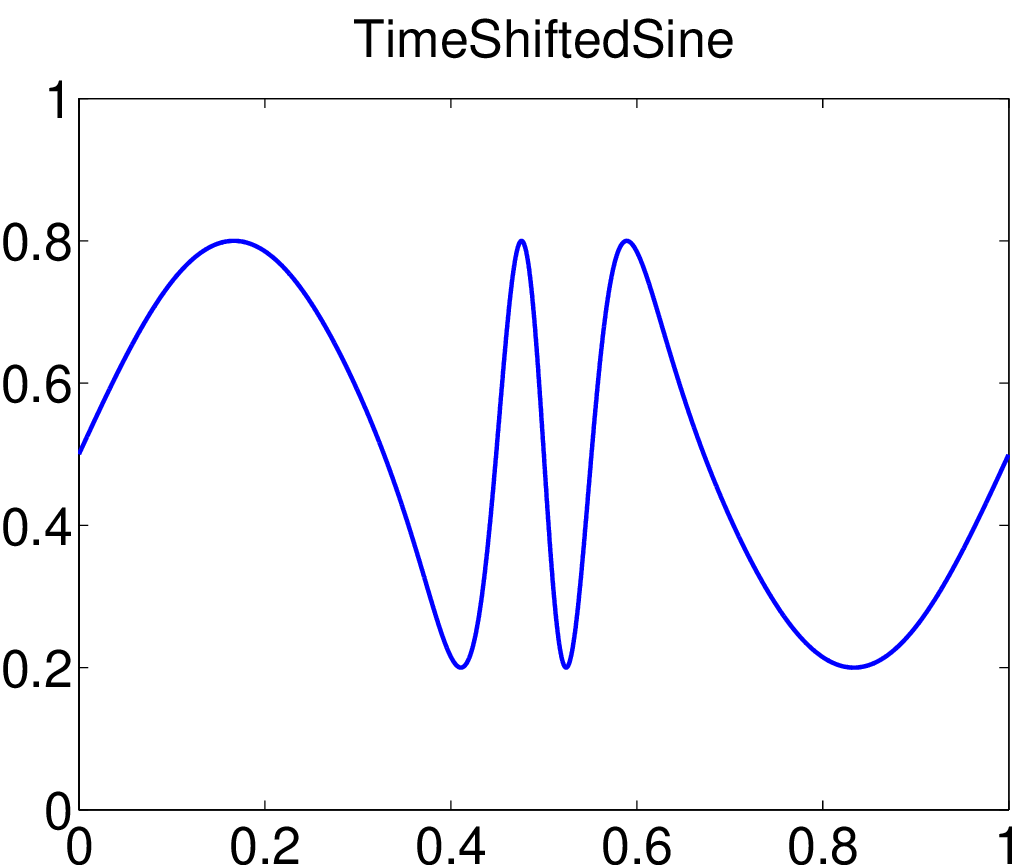} \\
\centerline{(c)}
\end{minipage}
 \end{tabular}
\caption{Original Signals   (a): Wave. (b): Parabolas.  (c): TimeShiftedSine.}
\label{fig:orig}
\end{figure}
We have compared the numerical performance of BlockJS to state-of-the-art classical thresholding methods of the literature. In particular we consider the block estimator of \cite{penskyyy} and two term-by-term thresholding methods. 
The first one is the classical hard thresholding and the other one corresponds to the non-negative garrote (introduced in wavelet estimation by \cite{Gao}). In the sequel, we name the estimator of \cite{penskyyy} by 'BlockH', the one of \cite{Gao} by 'TermJS' and our estimator by 'BlockJS'. For numerical implementation, the test functions were finely discretized by taking $T$ equispaced samples $t_i=i/T \in [0,1]$, $i=0,\ldots,T-1$. The deconvolution estimator was efficiently implemented in the Fourier domain given that Meyer wavelets are frequency band-limited. The performance of the estimators are measured in terms of peak signal-to-noise ratio (PSNR~$=10\log_{10}\frac{\max_{t_i \in [0,1]} |f^{(d)}(t_i)|^2}{\sum_{i=0}^{T-1} (f^{(d)}(t_i)- f^{(d)}(t_i))^2/T}$) in decibels (dB)). For any $v \in \{1,\ldots,n\}$, the blurring function $g_v$ is that of Example~\ref{ex:lap} and was used throughout all experiments.

\subsection{Monochannel simulation}
As an example of homoscedastic monochannel reconstruction (i.e. $n=1$), we show in \textsc{Fig}~\ref{fig:monopsnr} estimates obtained using the BlockJS method from $T=4096$ equispaced samples generated according to \eqref{sous} with blurred signal-to-noise ratio (BSNR) of $25$ dB (BSNR~$=10\log_{10}\frac{\sum_{i=0}^{T-1}(f \star g_v)(t_i)^2}{T\epsilon^2}$ dB). For $d=0$, the results are very effective for each test function where the singularities are well estimated. The estimator does also a good job in estimating the first and second-order derivatives, although the estimation quality decreases as the order of the derivative increases. This is in agreement with the predictions of the minimaxity results. 
We then have compared the performance of BlockJS with BlockH. The blurred signals were corrupted by a zero-mean white Gaussian noise such that the BSNR ranged from $10$ to $40$ dB. The PSNR values averaged over 10 noise realizations are depicted in \textsc{Fig}~\ref{fig:monopsnrcomp} for $d=0$, $d=1$ and $d=2$ respectively. One can see that our BlockJS thresholding estimator produces quite accurate estimates of $f$, $f'$ and $f''$ for each test function. These results clearly show that our approach compares favorably to BlockH and that BlockJS has good adaptive properties over a wide range of noise levels in the monochannel setting. 

\begin{figure}[htp!]
\hspace*{-0.5cm}
\begin{tabular}{ccc}
\begin{minipage}{0.33\textwidth}
\includegraphics[width=\textwidth]{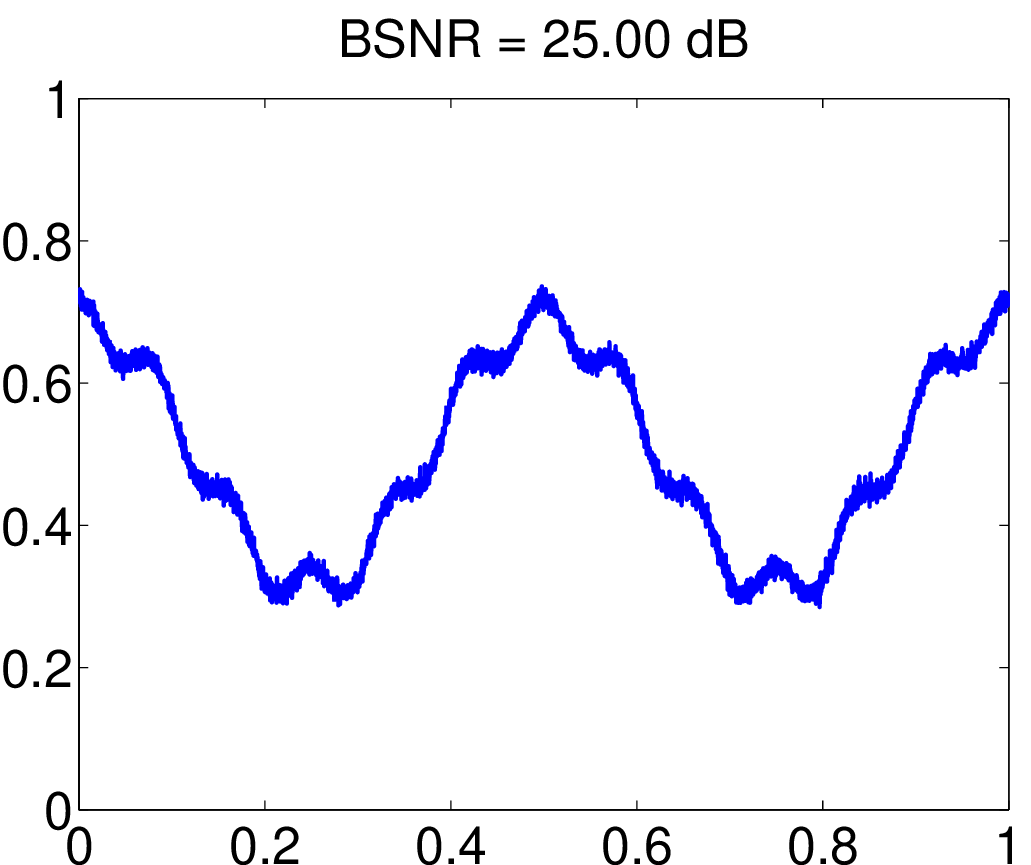}\\
\centerline{}
\end{minipage}&
\begin{minipage}{0.33\textwidth}
\includegraphics[width=\textwidth]{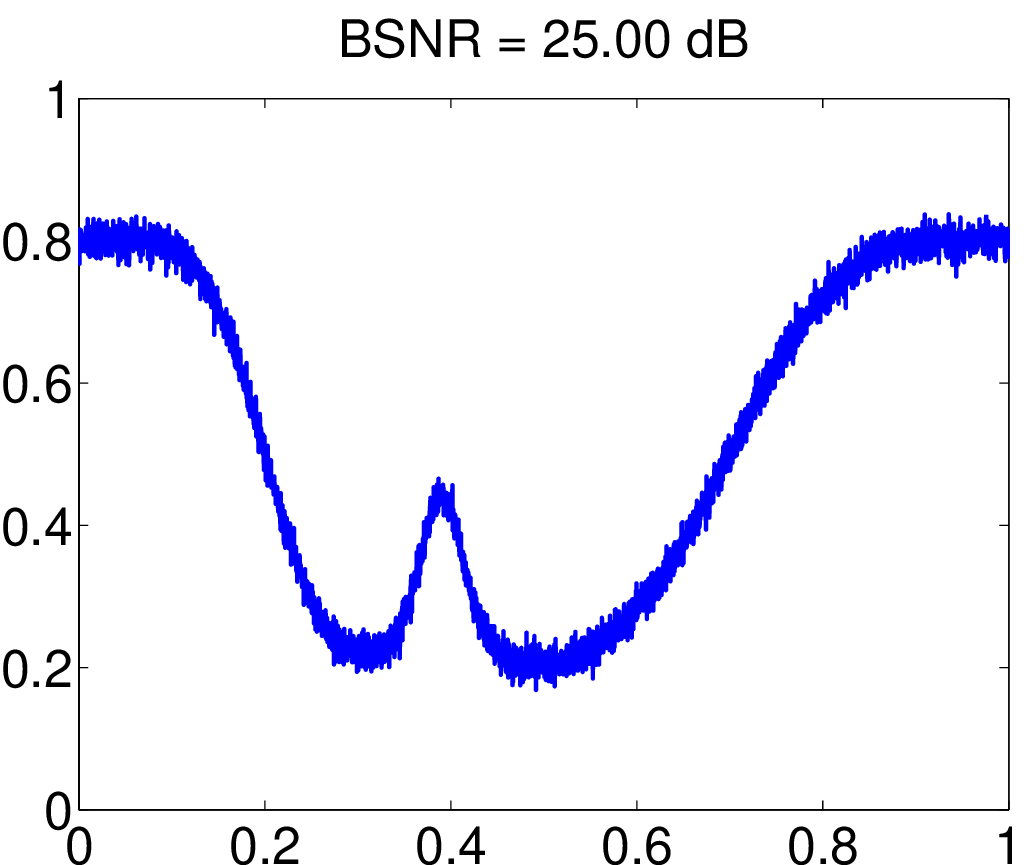}\\
\centerline{(a)}
\end{minipage}&
\begin{minipage}{0.33\textwidth}
\includegraphics[width=\textwidth]{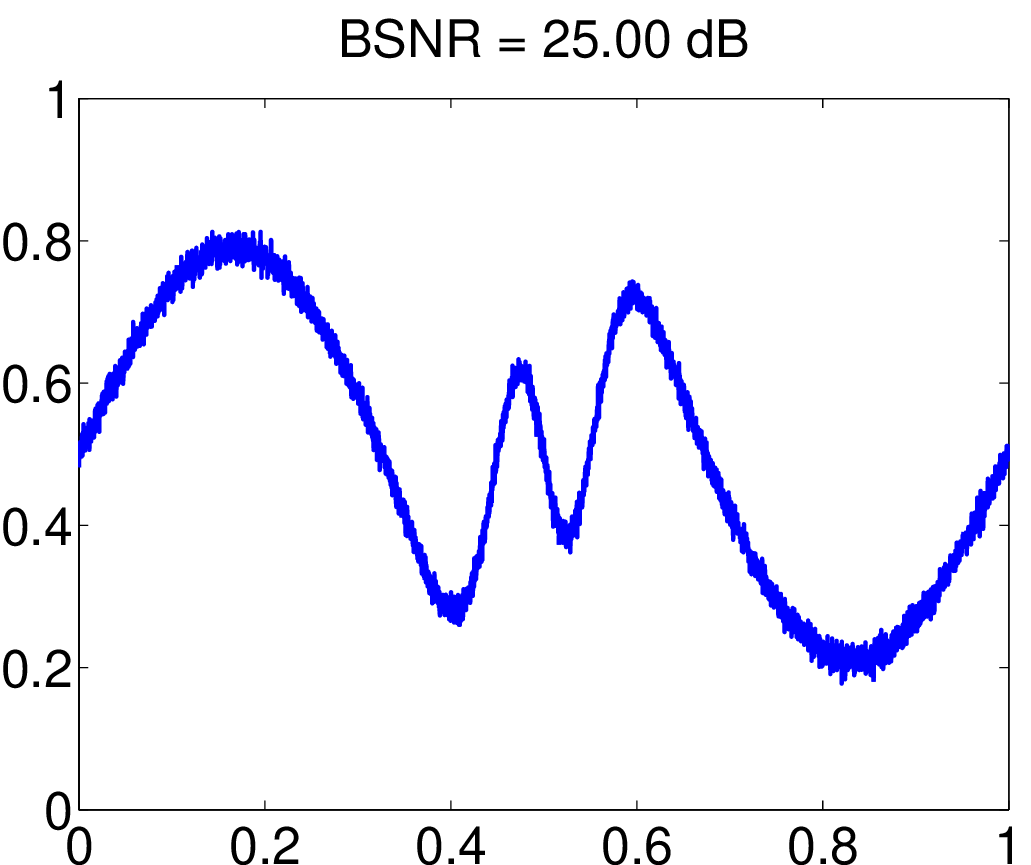}\\
\centerline{}
\end{minipage}
 \end{tabular}
\hspace*{-0.5cm}
\begin{tabular}{ccc}
\begin{minipage}{0.33\textwidth}
\includegraphics[width=\textwidth]{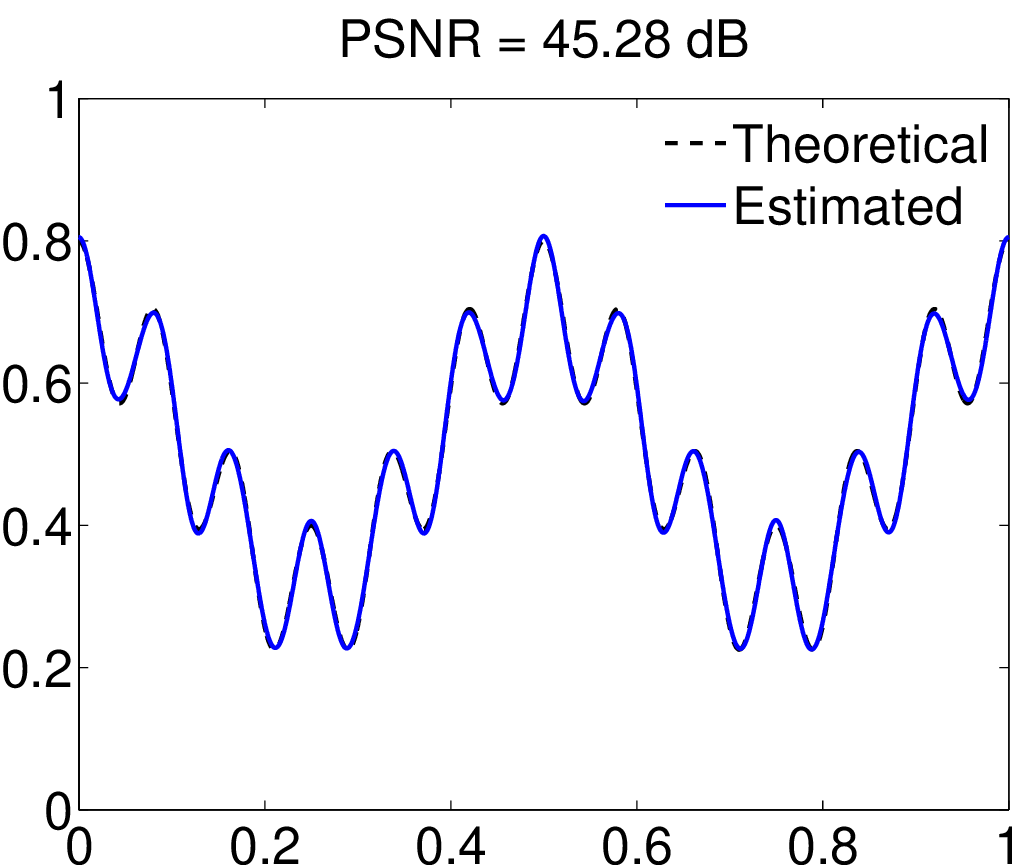}\\
\centerline{}
\end{minipage}&
\begin{minipage}{0.33\textwidth}
\includegraphics[width=\textwidth]{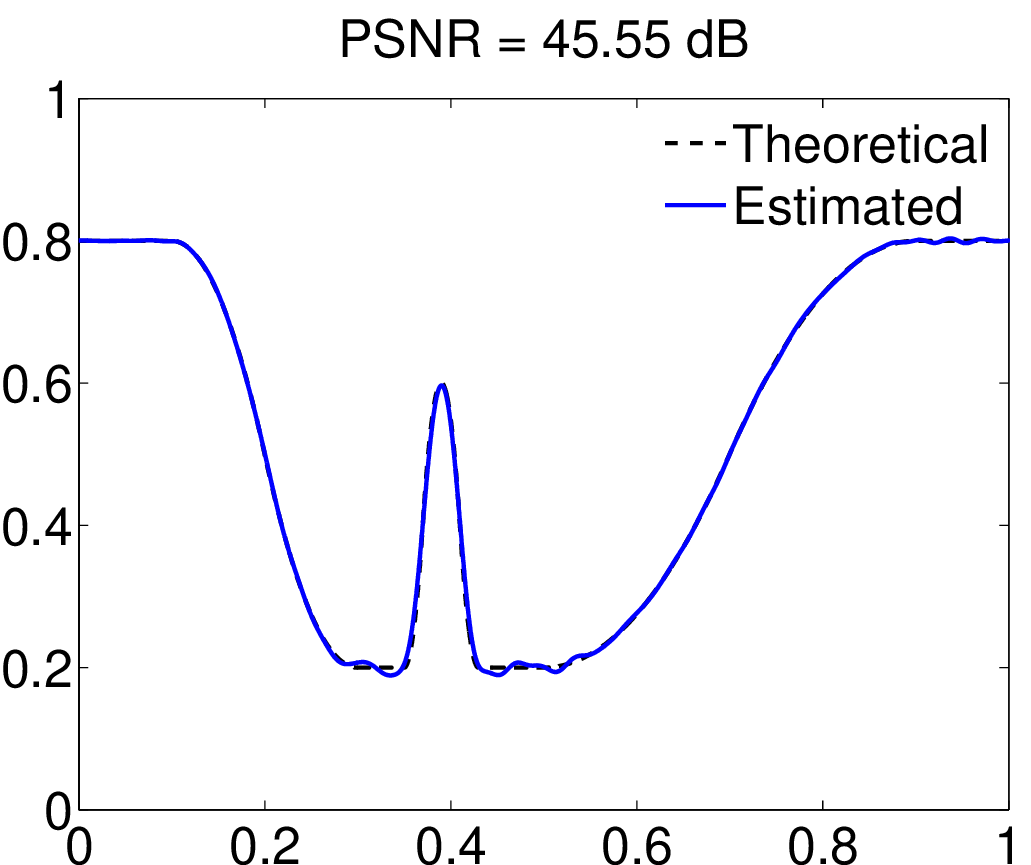}\\
\centerline{(b)~ $d=0$.}
\end{minipage}&
\begin{minipage}{0.33\textwidth}
\includegraphics[width=\textwidth]{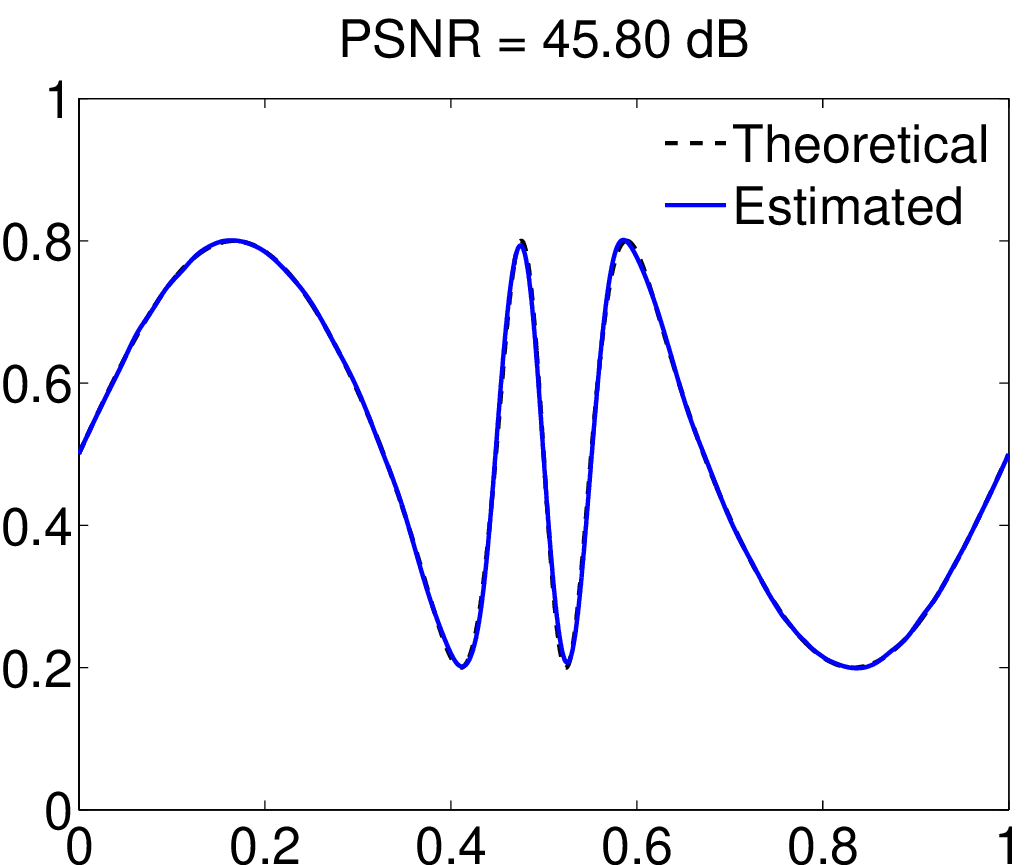}\\
\centerline{}
\end{minipage}
 \end{tabular}
\hspace*{-0.5cm}
\begin{tabular}{ccc}
\begin{minipage}{0.33\textwidth}
\includegraphics[width=\textwidth]{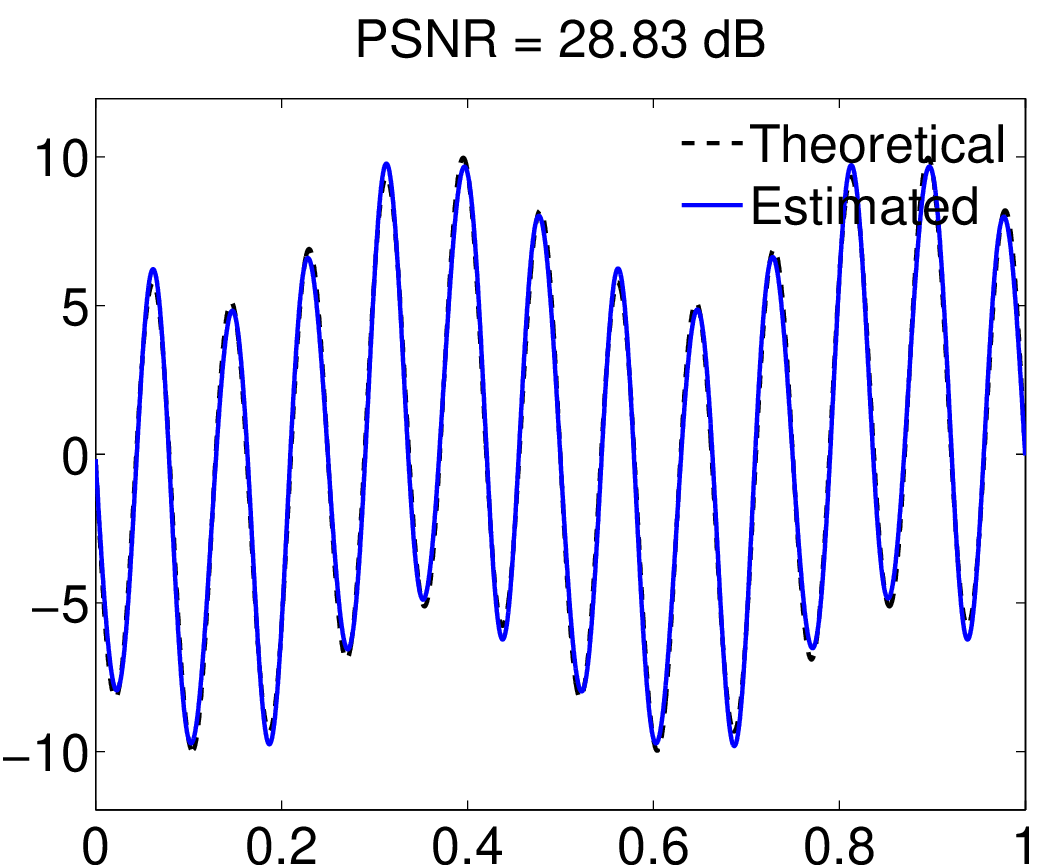}\\
\centerline{}
\end{minipage}&
\begin{minipage}{0.33\textwidth}
\includegraphics[width=\textwidth]{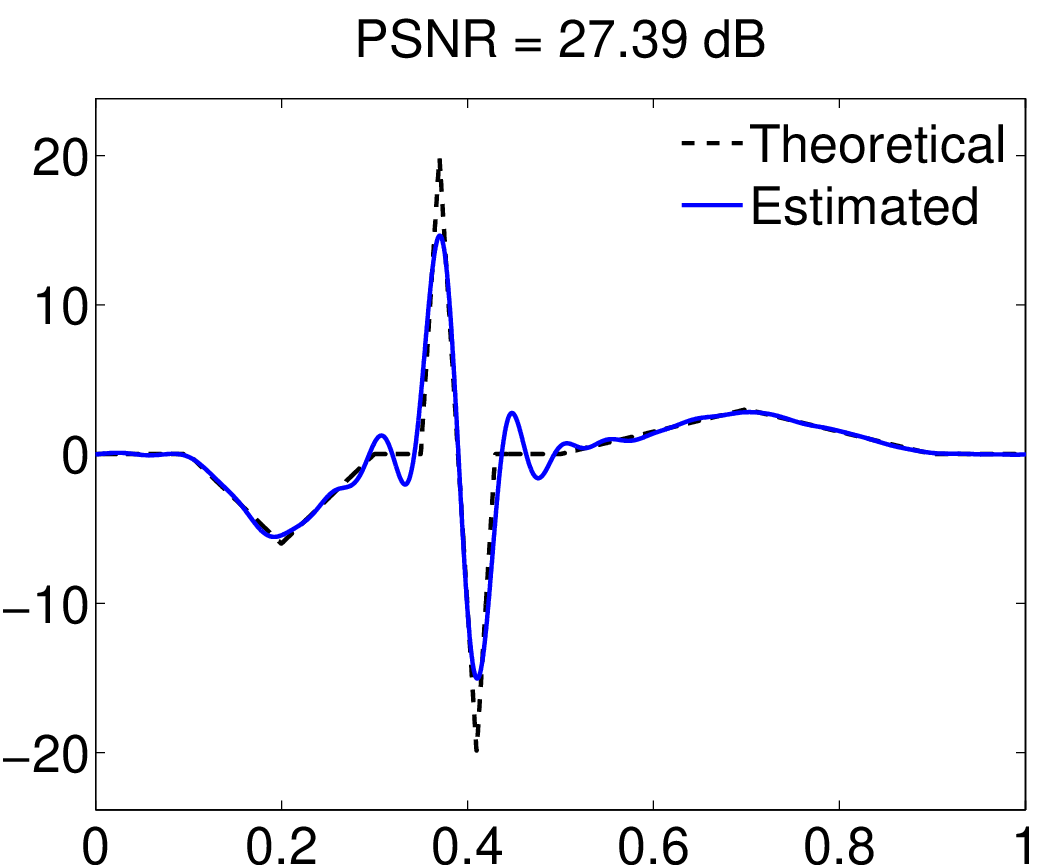}\\
\centerline{(c)~ $d=1$.}
\end{minipage}&
\begin{minipage}{0.33\textwidth}
\includegraphics[width=\textwidth]{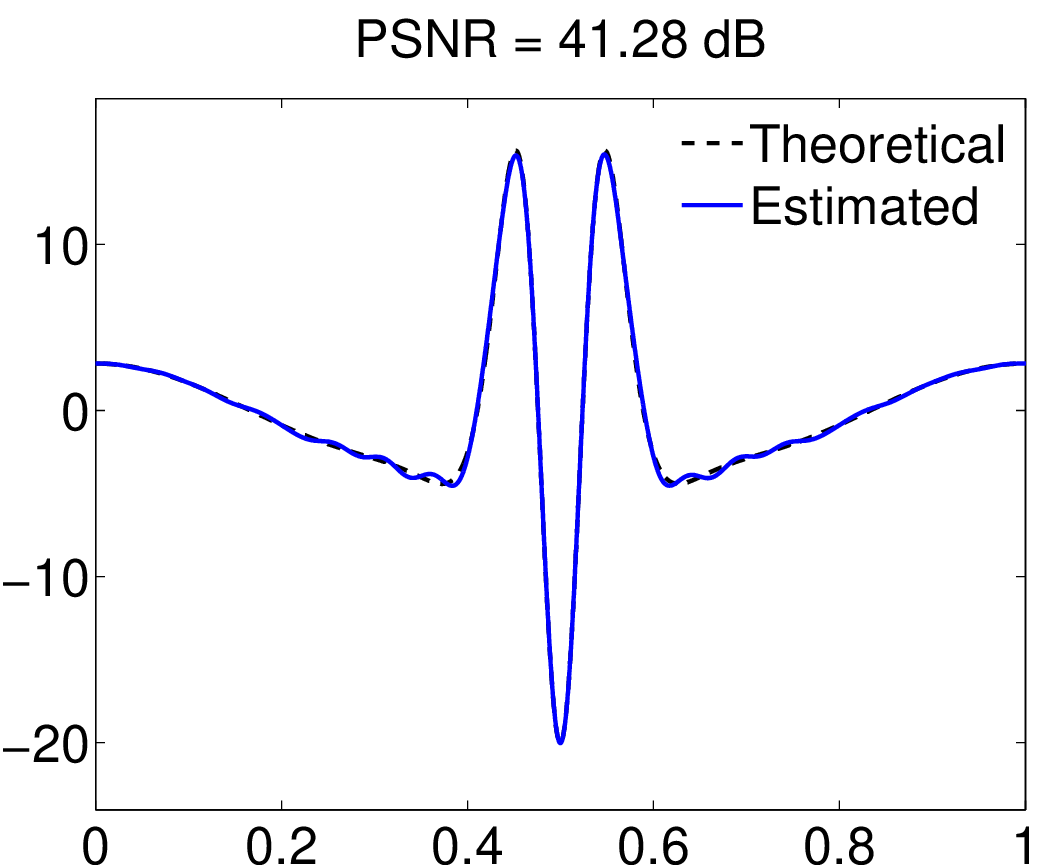}\\
\centerline{}
\end{minipage}
 \end{tabular}
\hspace*{-0.5cm}
\begin{tabular}{ccc}
\begin{minipage}{0.33\textwidth}
\includegraphics[width=\textwidth]{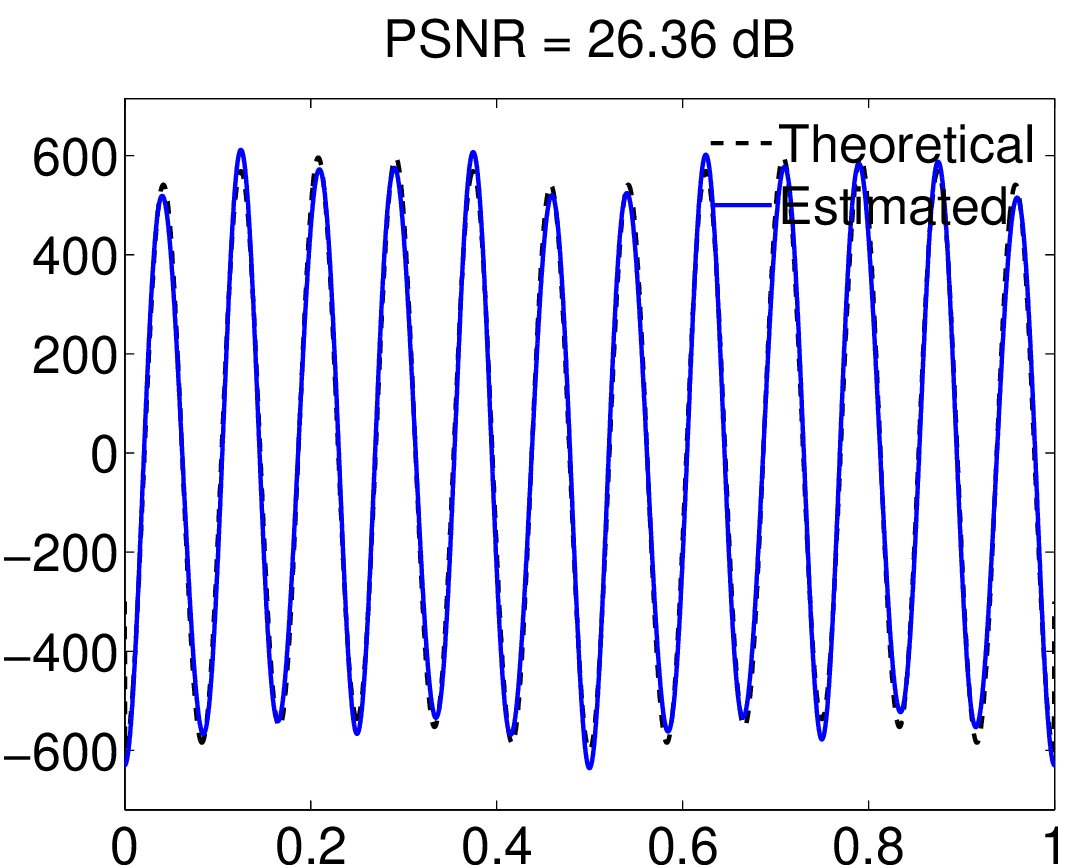}\\
\centerline{}
\end{minipage}&
\begin{minipage}{0.33\textwidth}
\includegraphics[width=\textwidth]{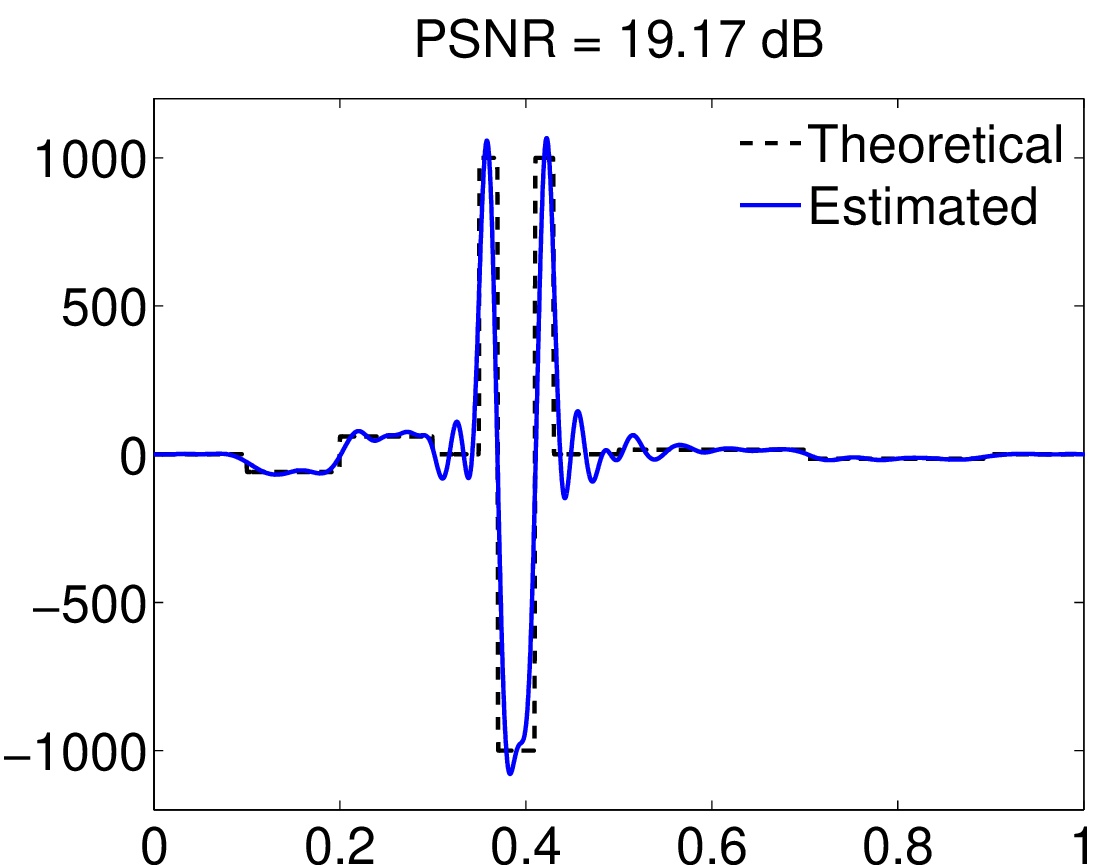}\\
\centerline{(d) ~ $d=2$.}
\end{minipage}&
\begin{minipage}{0.33\textwidth}
\includegraphics[width=\textwidth]{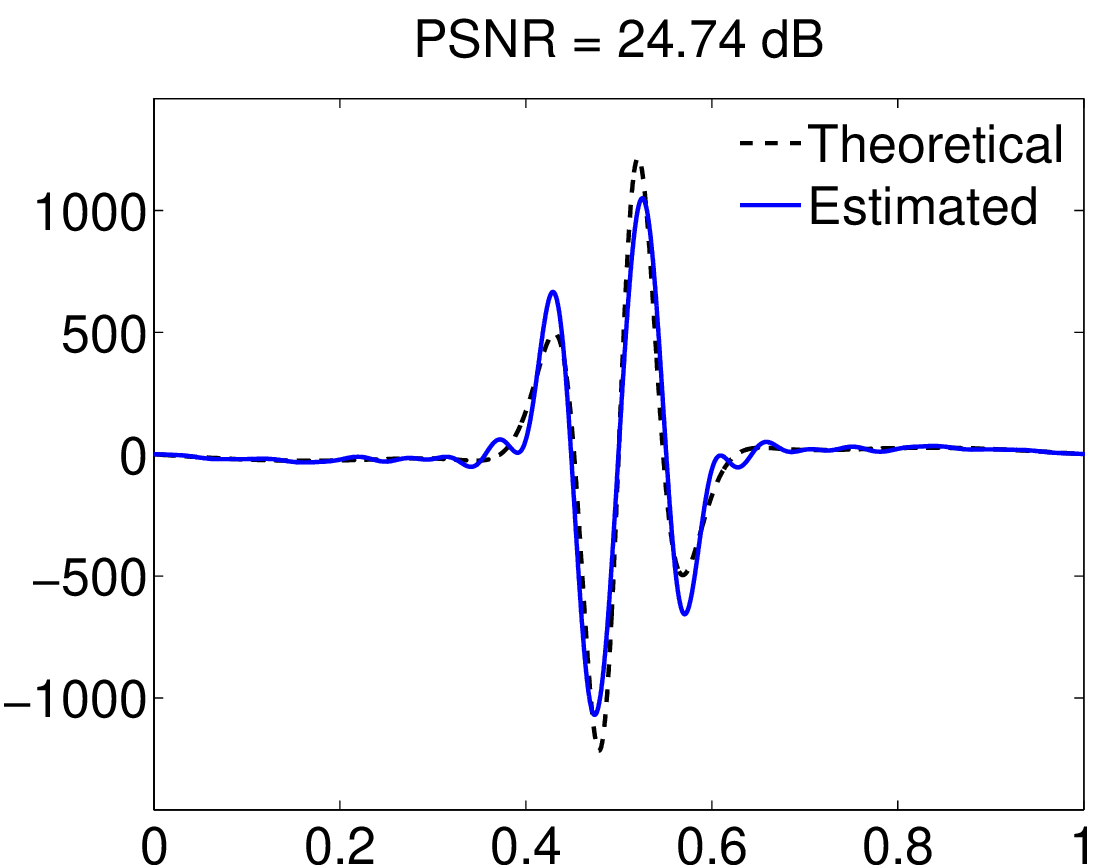}\\
\centerline{}
\end{minipage}
 \end{tabular}
\caption{Original (dashed) and estimated function/derivatives (solid) using the BlockJS estimator applied to noisy blurred observations shown in (a). (b): $d=0$. (c): $d=1$ (d): $d=2$. From left to right Wave, Parabolas and TimeShiftedSine.}
\label{fig:monopsnr}
\end{figure}

\begin{figure}[htp!]
 \centerline{$d=0$}
\hspace*{-0.5cm}
\begin{tabular}{ccc}
\begin{minipage}{0.33\textwidth}
\includegraphics[width=\textwidth]{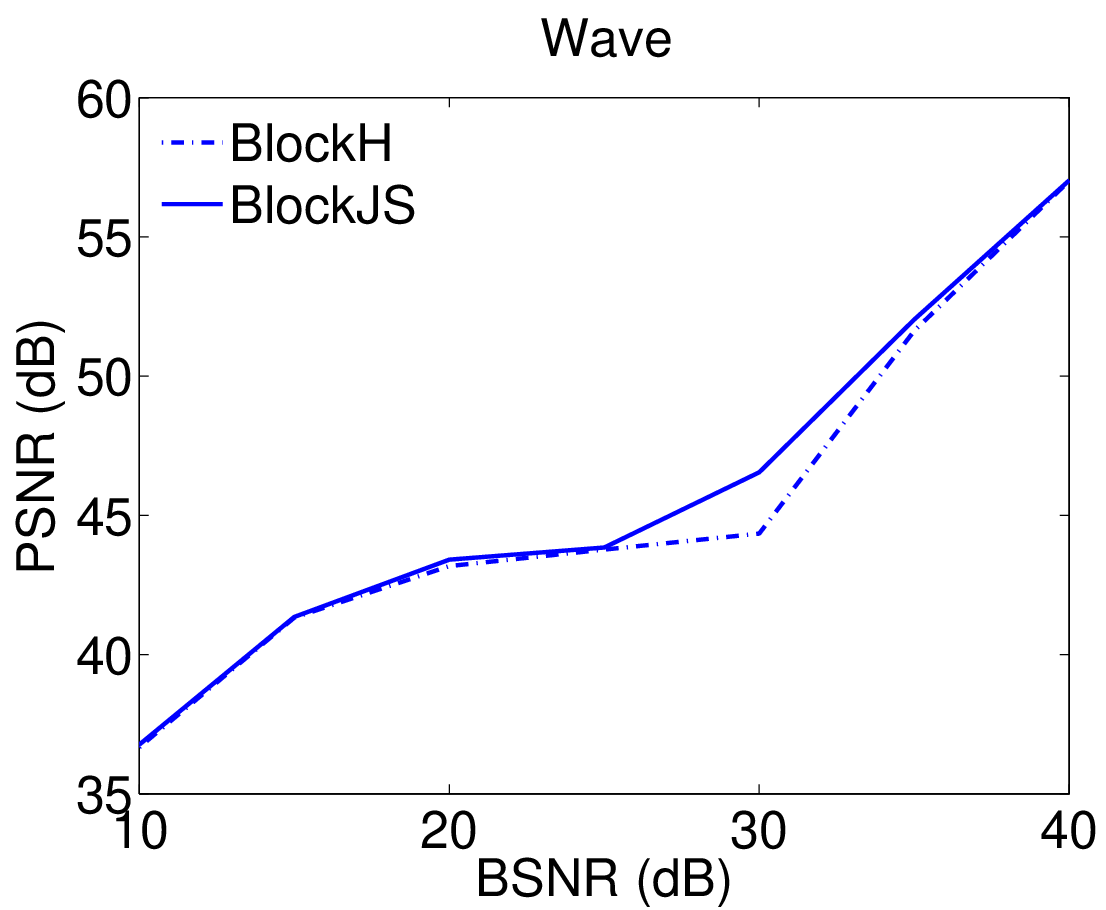} 
\end{minipage}&
\begin{minipage}{0.33\textwidth}
\includegraphics[width=\textwidth]{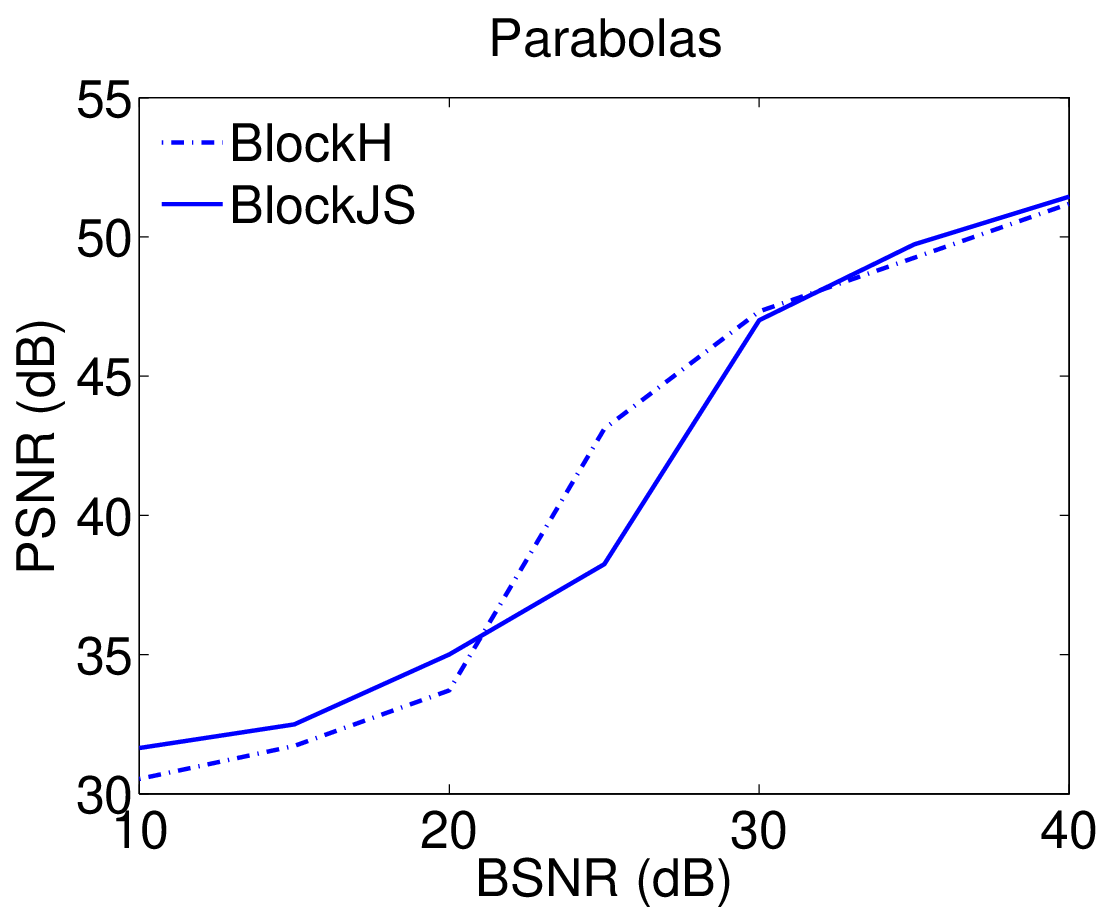} 
\end{minipage}&
\begin{minipage}{0.33\textwidth}
\includegraphics[width=\textwidth]{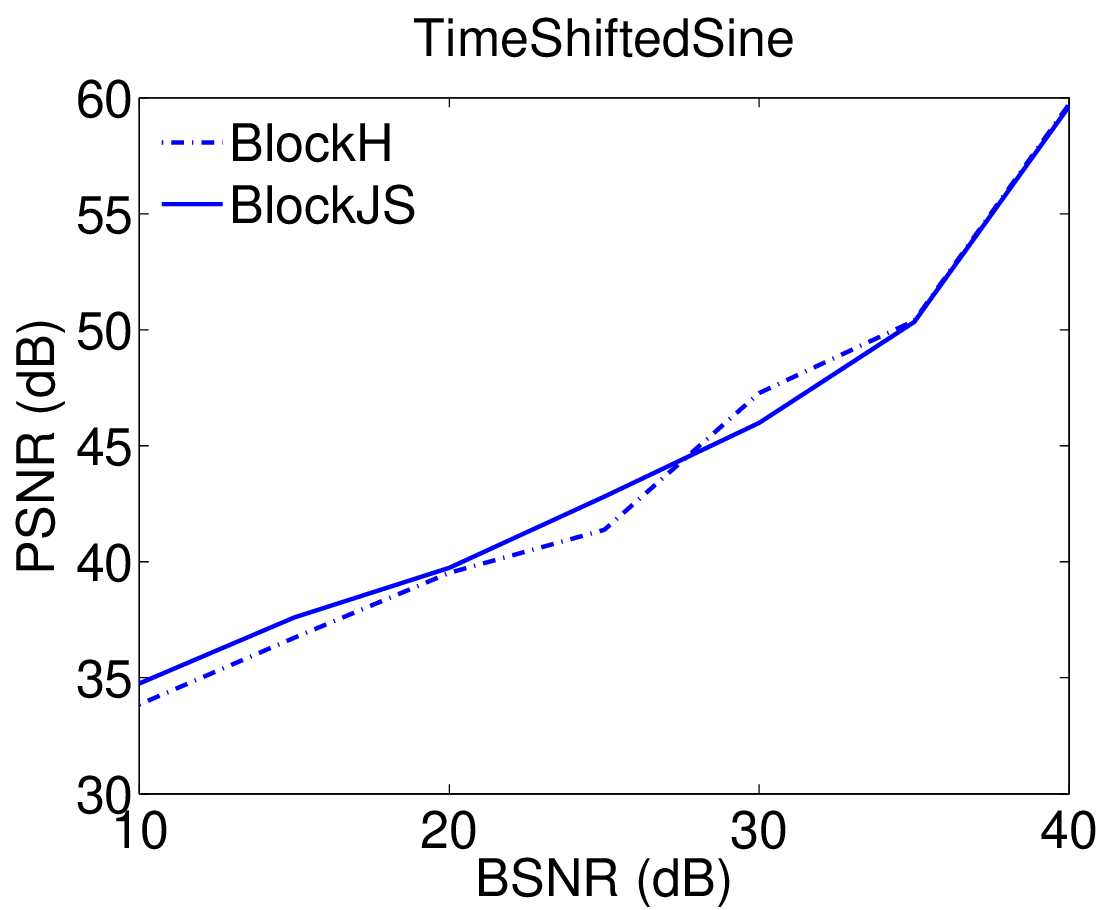}
\end{minipage}
\end{tabular}
\centerline{$d=1$}
\hspace*{-0.5cm}
\begin{tabular}{ccc}
\begin{minipage}{0.33\textwidth}
\includegraphics[width=\textwidth]{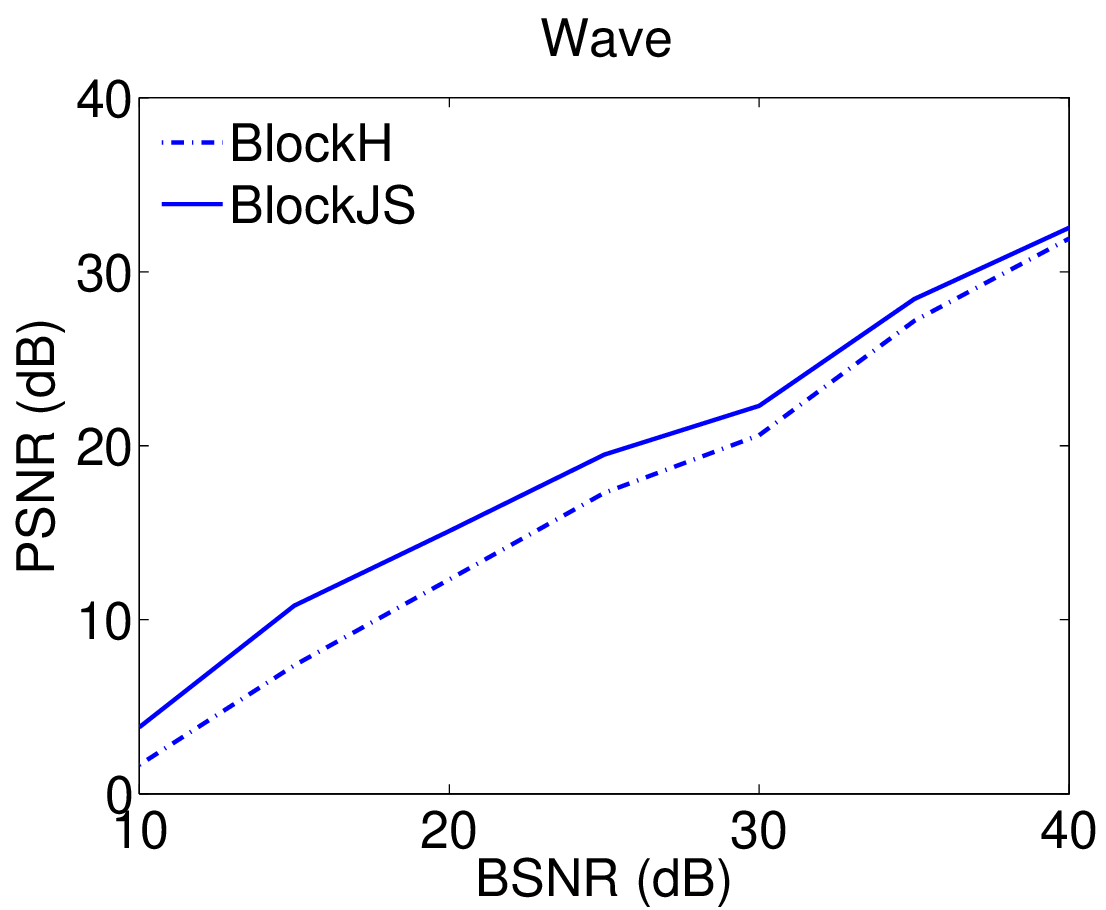} 
\end{minipage}&
\begin{minipage}{0.33\textwidth}
\includegraphics[width=\textwidth]{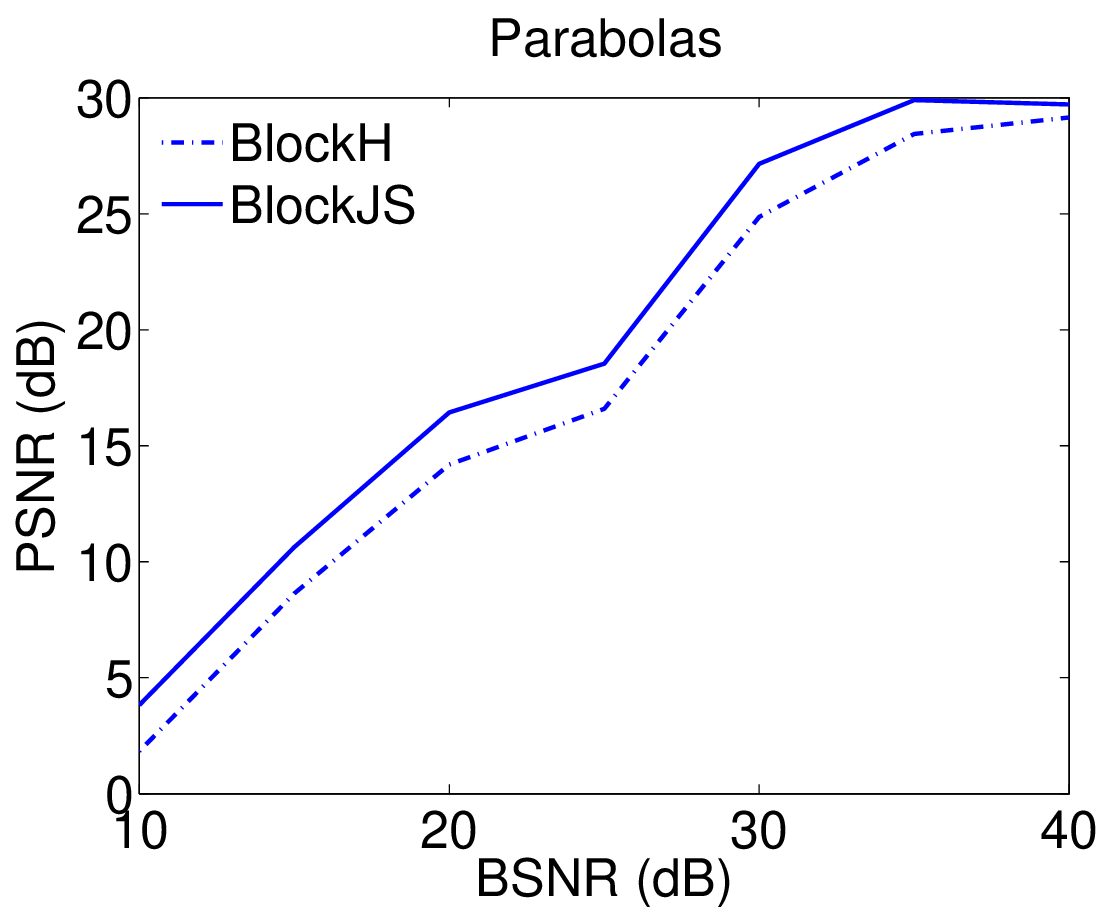} 
\end{minipage}&
\begin{minipage}{0.33\textwidth}
\includegraphics[width=\textwidth]{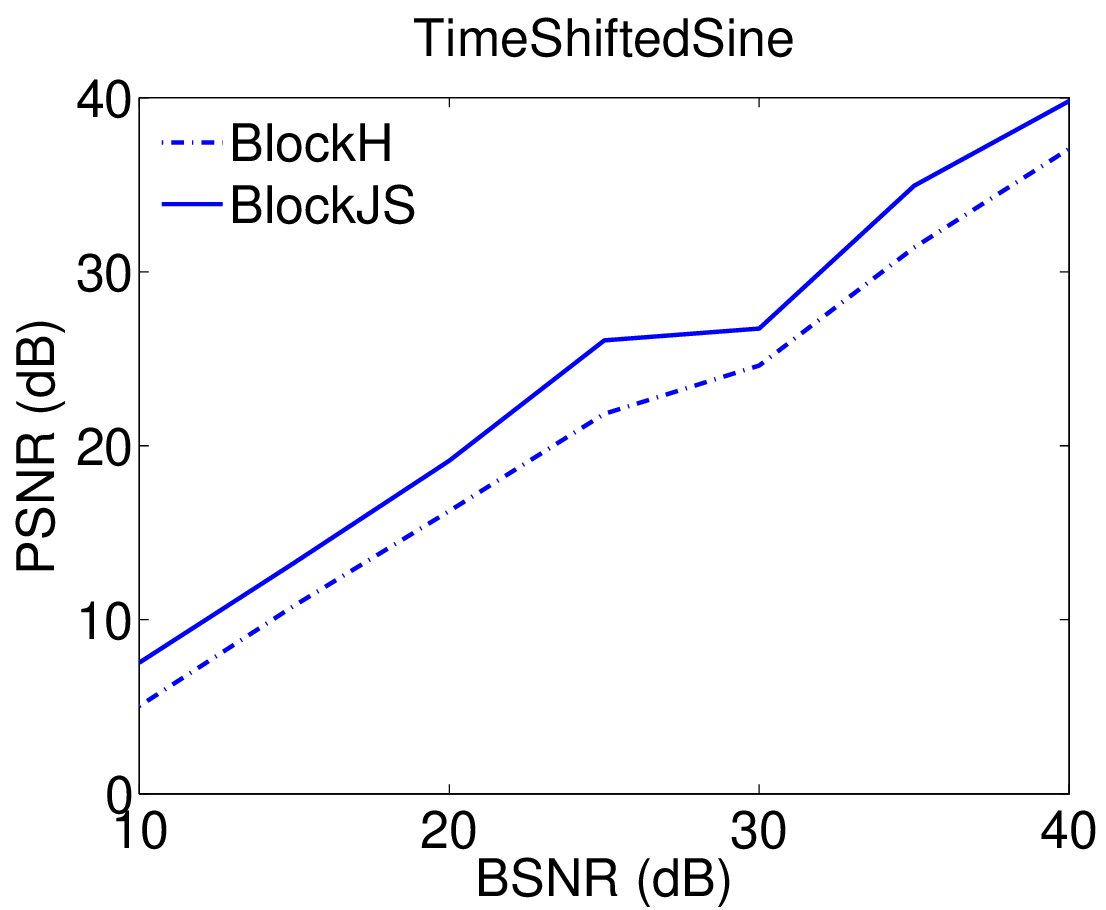}
\end{minipage}
\end{tabular}
\centerline{$d=2$}
\hspace*{-0.5cm}
\begin{tabular}{ccc}
\begin{minipage}{0.33\textwidth}
\includegraphics[width=\textwidth]{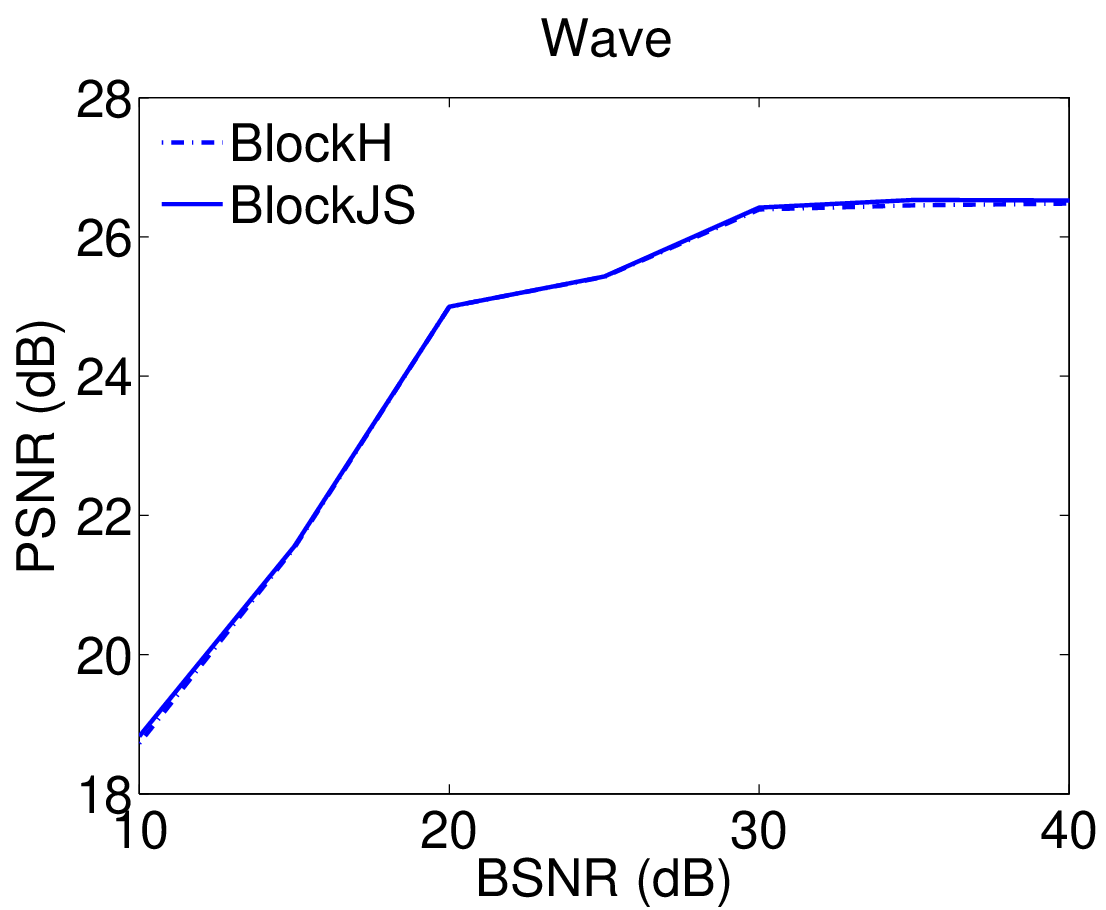} \\
\centerline{(a)}
\end{minipage}&
\begin{minipage}{0.33\textwidth}
\includegraphics[width=\textwidth]{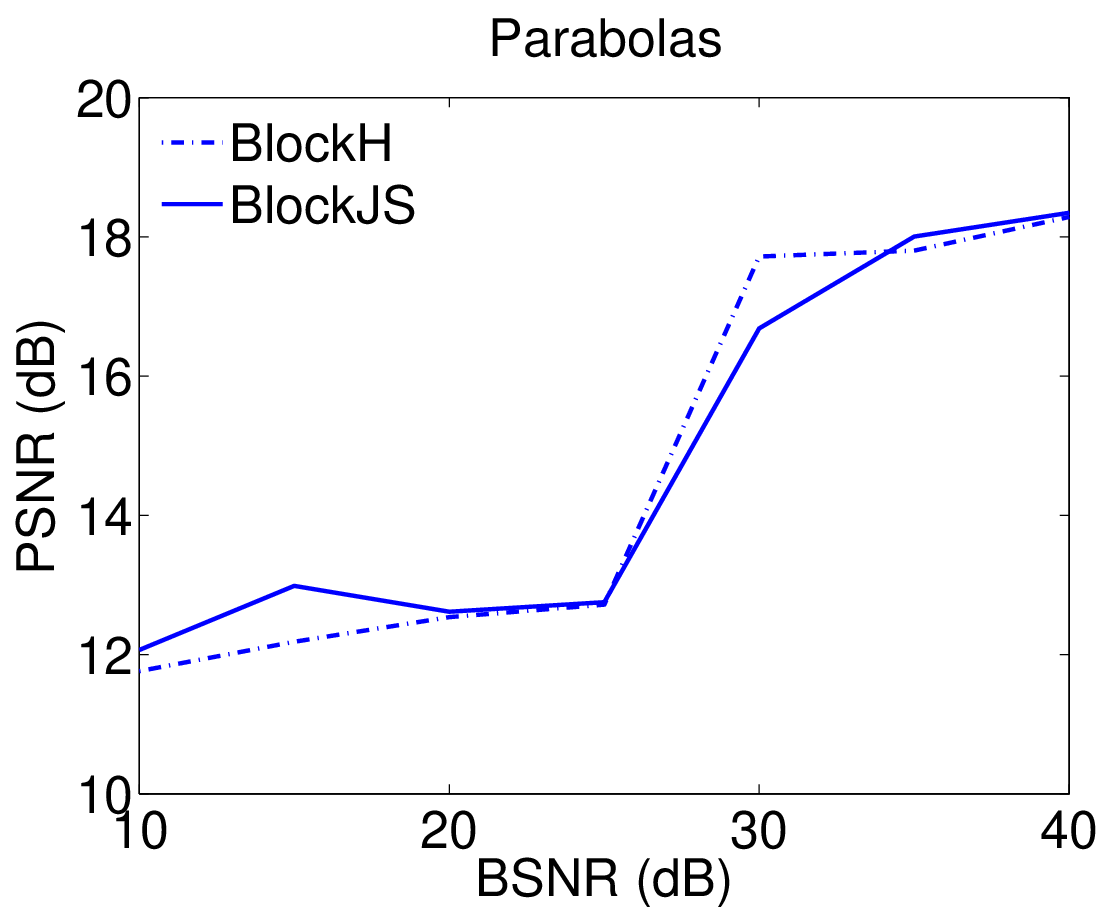} \\
\centerline{(b)}
\end{minipage}&
\begin{minipage}{0.33\textwidth}
\includegraphics[width=\textwidth]{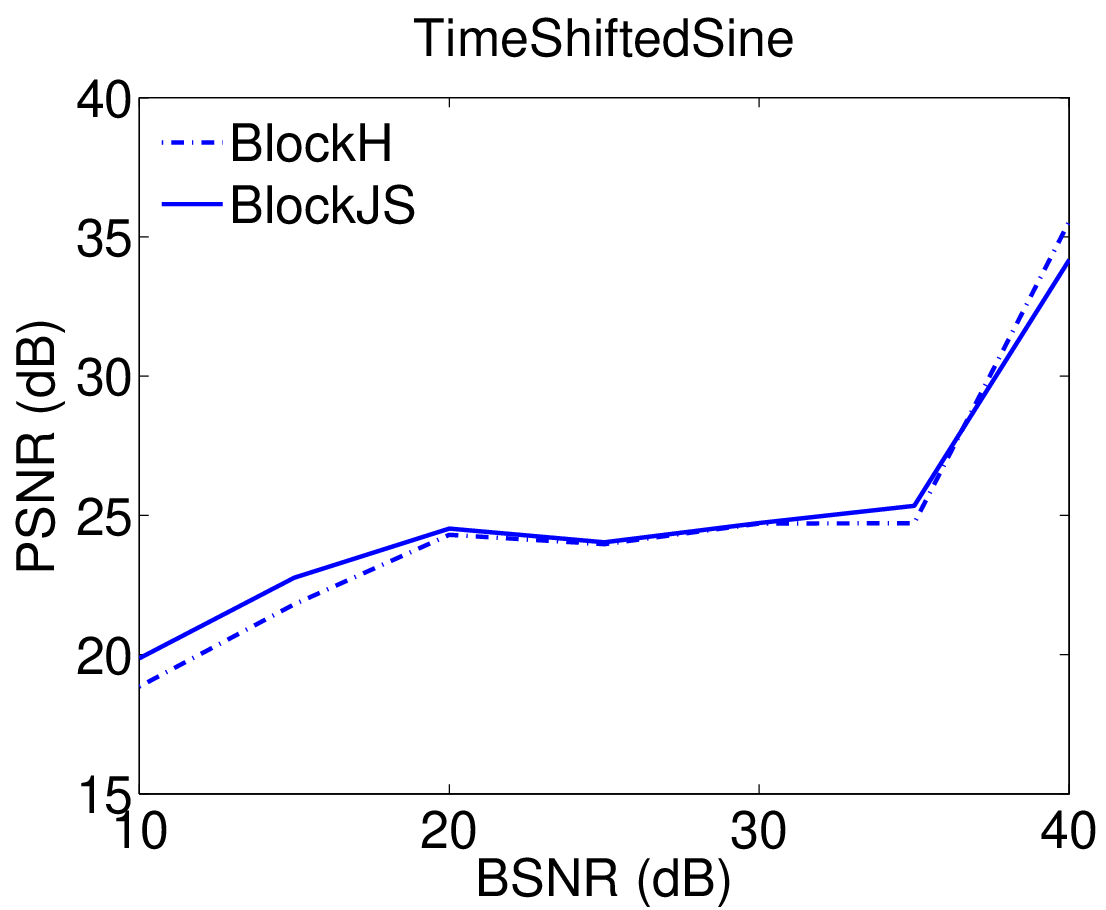}\\
\centerline{(c)}
\end{minipage}
\end{tabular}
\caption{Averaged PSNR values as a function of the input BSNR from 10 replications of the noise. (a): Wave. (b): Parabolas. (c): TimeShiftedSine. From top to bottom $d=0,1,2$.}
\label{fig:monopsnrcomp}
\end{figure} 

\subsection{Multichannel simulation}

A first point we would like to highlight is the fact that some choices of $\sigma_1,\dots,\sigma_n$ can severely impact the performance of the estimators. To illustrate this, we show in \textsc{Fig}~\ref{fig:sigv} an example of first derivative estimates obtained using BlockJS from $n=10$ channels with $T=4096$ samples and noise level corresponding to BSNR$=25$ dB, for $\sigma_v = v$ (dashed blue) and $\sigma_v$ randomly generated in $(0,+\infty)$ (solid blue). With $\sigma_v$ randomly generated, we can observe a significant PSNR improvement up to $6.85$ dB for the first derivative of TimeShiftedSine. Note that this improvement is marginal (about $0.60$ dB) for the most regular test signal (i.e. Wave). 
\begin{figure}
\hspace*{-0.5cm}
\begin{tabular}{ccc}
\begin{minipage}{0.33\textwidth}
\includegraphics[width=\textwidth]{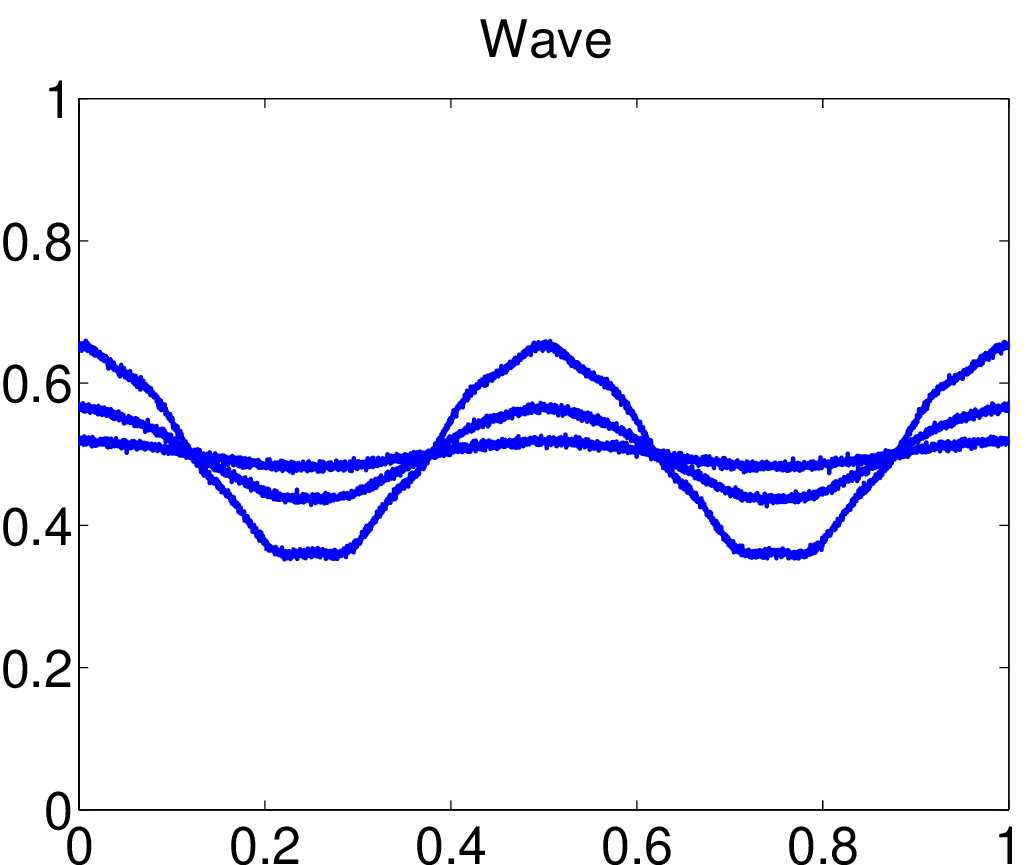}\\
\centerline{(a)}
\end{minipage}&
\begin{minipage}{0.33\textwidth}
\includegraphics[width=\textwidth]{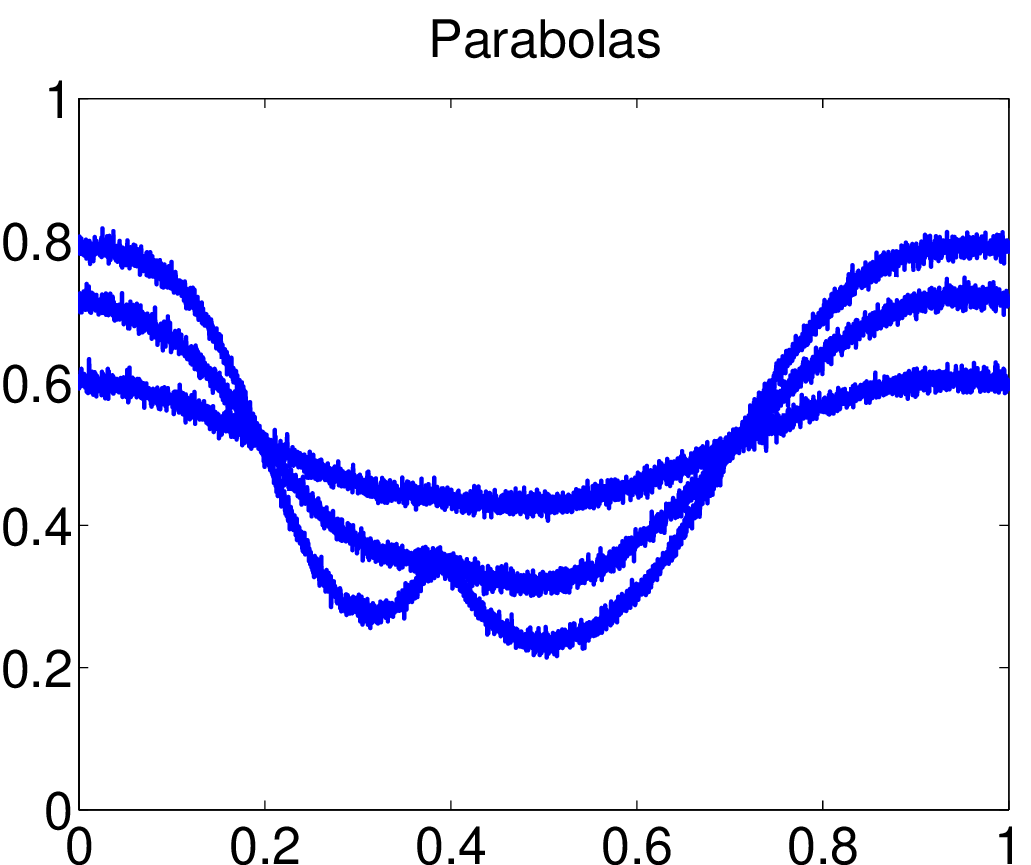}\\
\centerline{(b)}
\end{minipage}&
\begin{minipage}{0.33\textwidth}
\includegraphics[width=\textwidth]{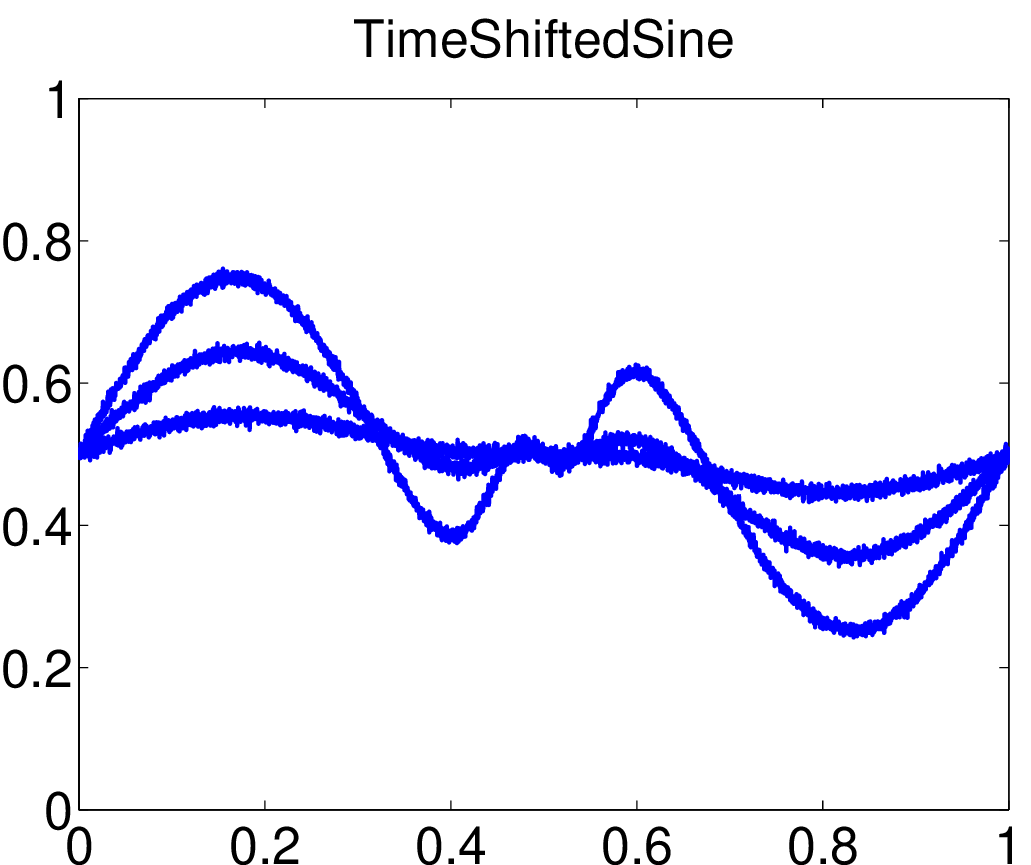}\\
\centerline{(c)}
\end{minipage}\\
\begin{minipage}{0.33\textwidth}
\includegraphics[width=\textwidth]{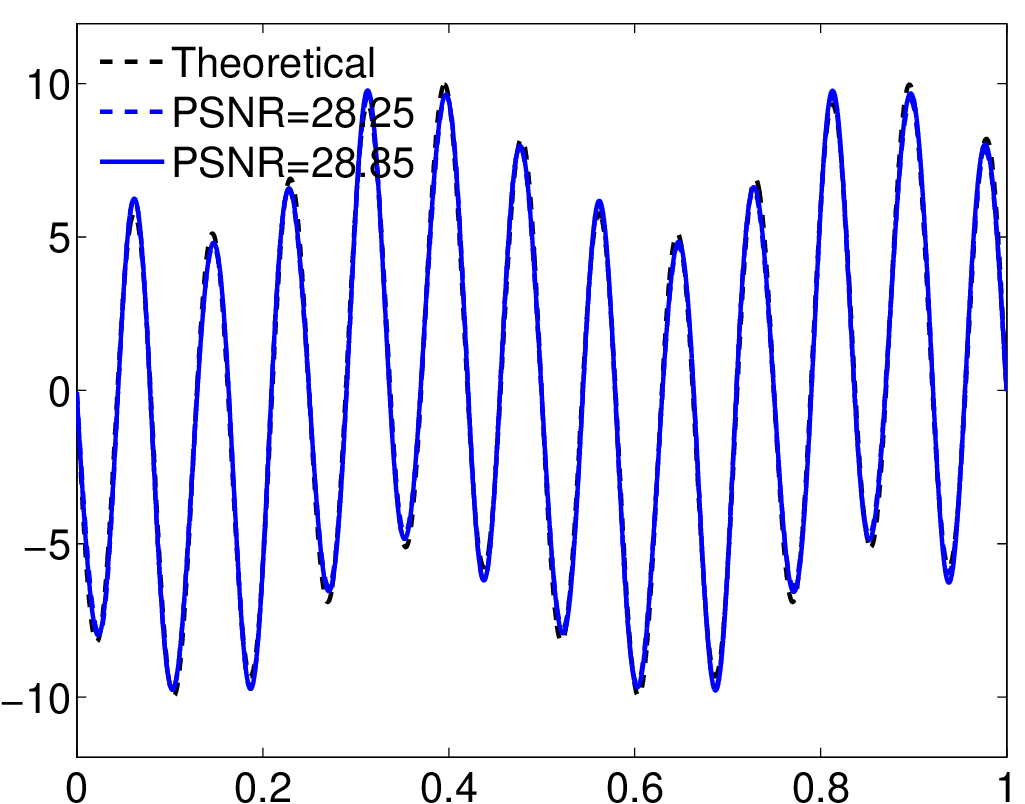} \\
\centerline{(d)}
\end{minipage}&
\begin{minipage}{0.33\textwidth}
\includegraphics[width=\textwidth]{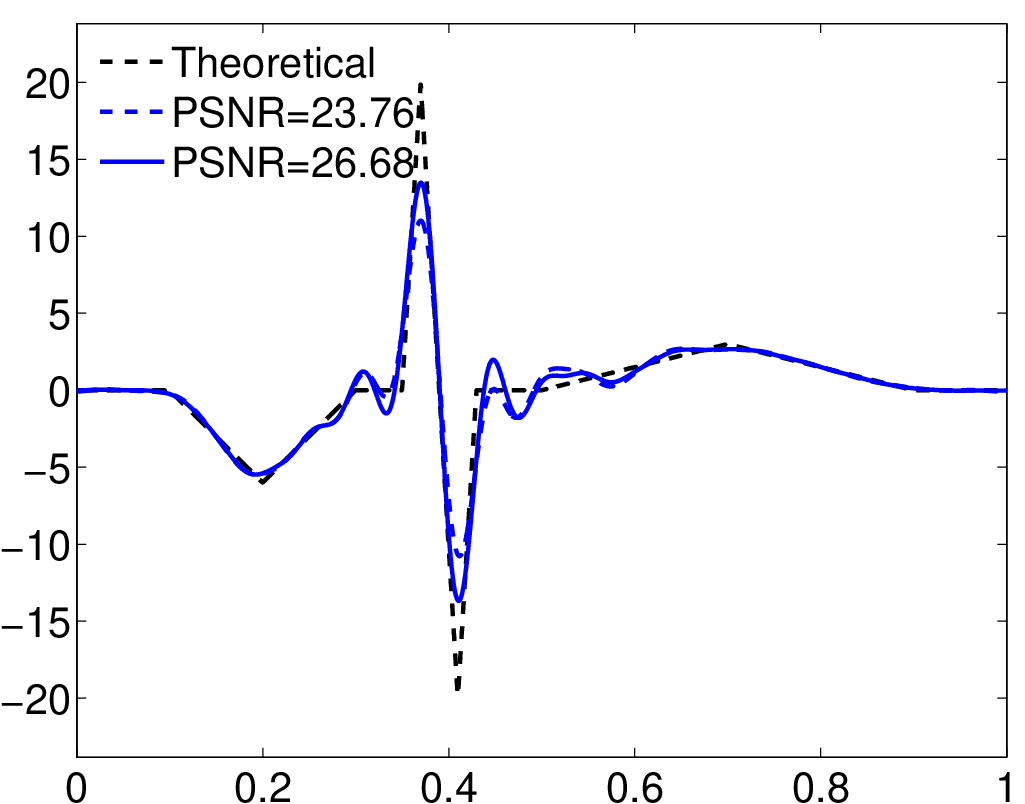} \\
\centerline{(e)}
\end{minipage}&
\begin{minipage}{0.33\textwidth}
\includegraphics[width=\textwidth]{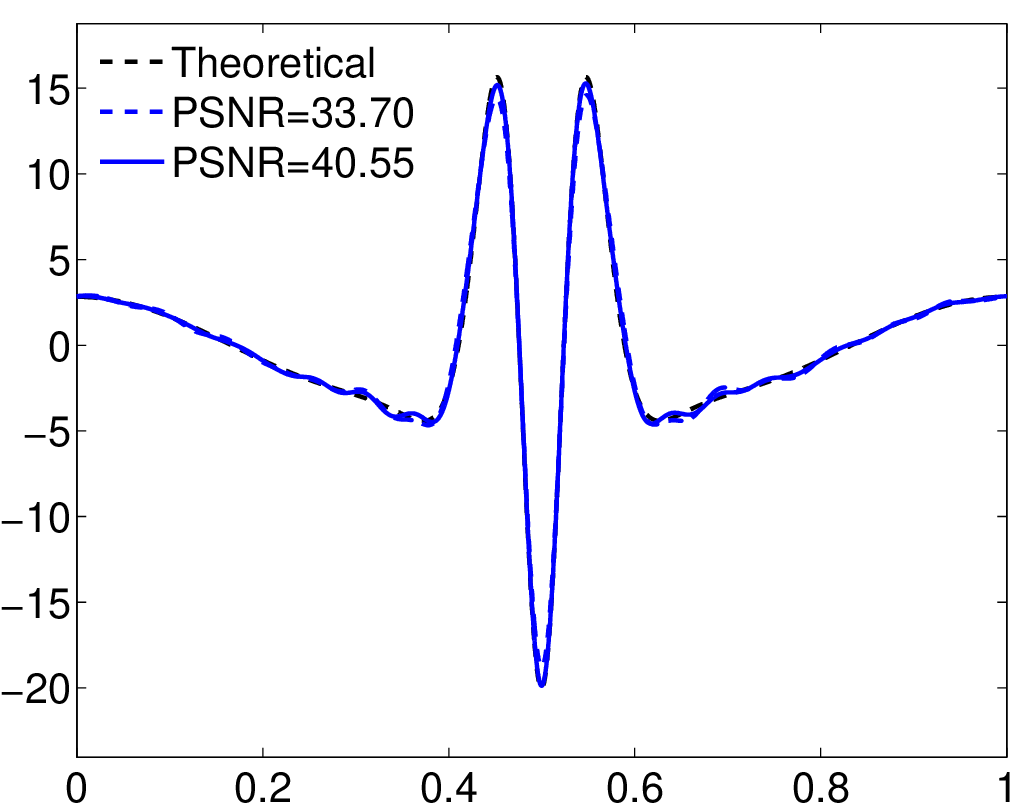} \\
\centerline{(f)}
\end{minipage}
\end{tabular}

\begin{tabular}{ccc}
\begin{minipage}{0.465\textwidth}
\includegraphics[width=\textwidth]{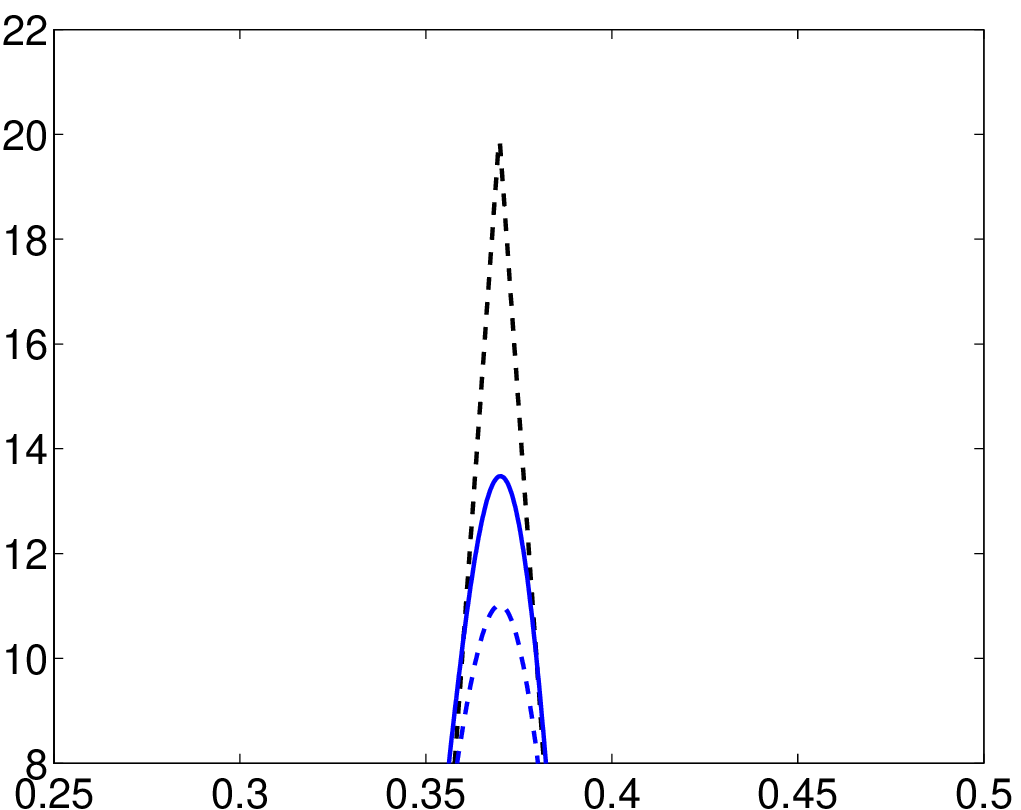} \\
\centerline{(g)}
\end{minipage}&
\begin{minipage}{0.465\textwidth}
\includegraphics[width=\textwidth]{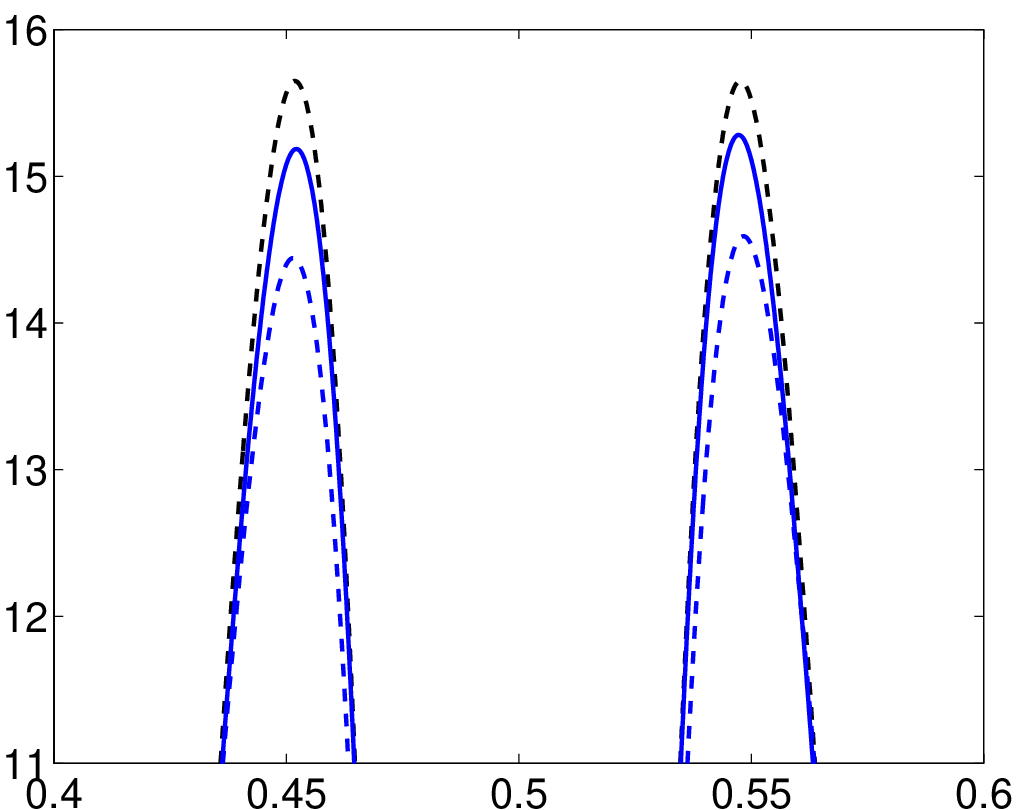} \\
\centerline{(h)}
\end{minipage}
\end{tabular}

\caption{Original functions (dashed black) and the estimate for $\sigma_v=v$ (dashed blue) and $\sigma_v$ randomly generated (solid blue) with $n=10$ channels. (a)-(c): noisy blurred observations (3 channels out of 10 shown). (d)-(f) BlockJS estimates of the first derivative. Zoom on the estimates (g): Parabolas, (h): TimeShiftedSine.}
\label{fig:sigv}
\end{figure}

\begin{table}[htp!]
\hspace*{-1.5cm}
\begin{tabular}{lc|c|c|c|c|c|c|c|c|c|c|c|c|c|}
\multicolumn{13}{c}{BSNR$=40$ dB}\\ \cline{2-13}
 & \multicolumn{4}{|c}{$d=0$} & \multicolumn{4}{|c}{$d=1$}& \multicolumn{4}{|c|}{$d=2$} \\ \cline{2-13}
  \hline
\multicolumn{1}{||l|}{$n$}      &  $10$ &  $20$ &  $50$ & $100$   &    $10$ &  $20$ &  $50$ &  $100$&    $10$ &  $20$ &  $50$ &  $100$\\
   \hline\hline 
 \multicolumn{13}{||c|}{\textit{Wave}}\\\hline
 \multicolumn{1}{||l|}{BlockJS} &$57.42$ &$66.40$& $66.83$& $74.62$&$42.64$ &$43.58$& $43.94$& $50.10$&$22.04$ &$30.34$& $33.88$& $36.69$\\
 \hline  
 \multicolumn{1}{||l|}{BlockH}  &$57.34$ &$66.31$& $66.78$& $74.72$&$42.43$ &$43.57$& $43.07$& $50.07$&$22.66$ &$29.14$& $33.78$& $36.65$\\
 \hline 
 \multicolumn{1}{||l|}{TermJS}  &$52.51$ &$61.94$& $64.86$& $73.57$&$40.64$ &$41.56$& $33.90$& $49.30$&$19.17$ &$29.57$& $30.87$& $33.20$ \\
 \hline  
 \multicolumn{1}{||l|}{TermH}   &$50.97$ &$52.72$& $55.69$& $74.73$&$31.39$ &$33.92$& $37.20$& $39.48$&$17.17$ &$28.82$& $33.89$& $35.61$\\
 \hline 
 \multicolumn{13}{||c|}{\textit{TimeShiftedSine}}\\\hline
 \multicolumn{1}{||l|}{BlockJS} &$65.11$ &$65.58$& $68.47$& $71.17$&$42.16$ &$46.25$& $46.85$& $49.53$&$41.09$ &$42.67$& $43.07$& $46.35$\\
 \hline  
 \multicolumn{1}{||l|}{BlockH}  &$62.11$ &$62.29$& $62.71$& $70.43$&$41.97$ &$45.57$& $45.75$& $49.35$&$39.20$ &$40.55$& $41.97$& $42.00$\\
 \hline 	
 \multicolumn{1}{||l|}{TermJS}  &$64.01$ &$64.73$& $66.13$& $68.21$&$40.96$ &$42.13$& $43.63$& $43.98$&$38.57$ &$41.17$& $41.73$& $45.15$\\
 \hline  
 \multicolumn{1}{||l|}{TermH}   &$61.84$ &$62.05$& $67.12$& $68.80$&$41.22$ &$42.12$& $43.32$& $44.69$&$39.39$ &$39.81$& $39.24$& $43.04$ \\
 \hline 
 \multicolumn{13}{||c|}{\textit{Parabolas}}\\\hline
 \multicolumn{1}{||l|}{BlockJS} &$56.18$ &$57.42$& $58.10$& $58.40$&$29.66$ &$29.76$& $31.04$& $31.40$&$20.63$&$21.88$& $21.93$& $21.90$\\
 \hline  
 \multicolumn{1}{||l|}{BlockH}  &$55.92$ &$57.25$& $57.70$& $58.10$&$29.69$ &$29.67$& $29.91$& $29.94$&$20.74$&$21.09$& $21.99$& $21.73$\\
 \hline 
 \multicolumn{1}{||l|}{TermJS}  &$54.33$ &$57.05$& $57.96$& $58.24$&$29.21$ &$29.38$& $29.65$& $29.79$&$20.57$&$20.93$& $20.91$& $21.51$\\
 \hline  
 \multicolumn{1}{||l|}{TermH}   &$54.59$ &$56.96$& $57.88$& $58.10$&$29.60$ &$29.60$& $29.90$& $29.94$&$20.51$&$20.96$& $20.84$& $21.49$\\
 \hline\\
\end{tabular}

\hspace*{-1.5cm}
\begin{tabular}{lc|c|c|c|c|c|c|c|c|c|c|c|c|c|}
\multicolumn{13}{c}{BSNR$=25$ dB}\\ \cline{2-13}
 & \multicolumn{4}{|c}{$d=0$} & \multicolumn{4}{|c}{$d=1$}& \multicolumn{4}{|c|}{$d=2$} \\ \cline{2-13}
  \hline
\multicolumn{1}{||l|}{$n$}      &  $10$ &  $20$ &  $50$ & $100$&    $10$ &  $20$ &  $50$ &  $100$&    $10$ &  $20$ &  $50$ &  $100$\\
   \hline\hline 
 \multicolumn{13}{||c|}{\textit{Wave}}\\\hline
 \multicolumn{1}{||l|}{BlockJS} &$44.04$ &$51.93$& $52.47$& $59.76$&$30.14$ &$30.90$& $30.90$& $35.93$&$26.83$&$26.83$&$26.98$&$27.00$\\
 \hline  
 \multicolumn{1}{||l|}{BlockH}  &$42.12$ &$51.35$& $51.82$& $59.73$&$28.69$ &$28.69$& $28.72$& $35.26$&$26.89$&$26.85$&$26.98$&$26.48$\\
 \hline 
 \multicolumn{1}{||l|}{TermJS}  &$41.67$ &$48.10$& $49.44$& $52.59$&$28.49$ &$28.69$& $28.72$& $28.71$&$25.25$&$26.62$&$26.78$&$26.94$\\
 \hline  
 \multicolumn{1}{||l|}{TermH}   &$40.95$ &$49.03$& $50.22$& $55.04$&$28.69$ &$28.69$& $28.72$& $28.72$&$25.28$&$26.85$&$26.98$&$26.86$\\
 \hline 
 \multicolumn{13}{||c|}{\textit{TimeShiftedSine}}\\\hline
 \multicolumn{1}{||l|}{BlockJS} &$51.85$ &$52.33$& $55.66$& $60.49$&$39.48$ &$41.46$& $41.88$& $41.92$ &$27.39$ &$28.72$& $29.05$& $35.95$\\
 \hline  
 \multicolumn{1}{||l|}{BlockH}  &$52.93$ &$51.51$& $55.82$& $60.35$&$38.68$ &$41.24$& $41.87$& $41.84$&$26.68$ &$29.36$& $29.54$& $35.15$\\
 \hline 
 \multicolumn{1}{||l|}{TermJS}  &$47.19$ &$47.83$& $54.45$& $56.63$&$29.46$ &$41.34$& $41.85$& $41.79$&$23.66$ &$25.84$& $25.95$& $27.58$\\
 \hline  
 \multicolumn{1}{||l|}{TermH}   &$47.54$ &$47.47$& $54.44$& $59.63$&$31.42$ &$41.03$& $40.75$& $41.79$&$23.66$ &$25.69$& $25.91$& $30.67$\\
 \hline 
 \multicolumn{13}{||c|}{\textit{Parabolas}}\\\hline
 \multicolumn{1}{||l|}{BlockJS} &$47.81$ &$49.88$& $52.90$& $54.40$&$25.52$ &$26.11$& $28.74$& $29.57$&$17.48$ &$18.24$& $18.96$& $20.62$\\
 \hline  
 \multicolumn{1}{||l|}{BlockH}  &$47.74$ &$49.23$& $52.00$& $53.62$&$25.43$ &$25.16$& $29.55$& $29.69$&$17.08$ &$18.92$& $19.00$& $20.63$\\
 \hline 
 \multicolumn{1}{||l|}{TermJS}  &$44.52$ &$49.80$& $50.80$& $52.79$&$24.84$ &$25.78$& $25.71$& $25.73$&$16.39$ &$18.58$& $19.00$& $20.57$\\
 \hline  
 \multicolumn{1}{||l|}{TermH}   &$43.74$ &$48.84$& $51.35$& $53.24$&$24.86$ &$25.17$& $25.10$& $27.98$&$16.49$ &$18.64$& $18.95$& $20.34$\\
 \hline\\
\end{tabular}

\hspace*{-1.5cm}
\begin{tabular}{lc|c|c|c|c|c|c|c|c|c|c|c|c|c|}
\multicolumn{13}{c}{BSNR$=10$ dB}\\ \cline{2-13}
 & \multicolumn{4}{|c}{$d=0$} & \multicolumn{4}{|c}{$d=1$}& \multicolumn{4}{|c|}{$d=2$} \\ \cline{2-13}
  \hline
\multicolumn{1}{||l|}{$n$}      &  $10$ &  $20$ &  $50$ & $100$&    $10$ &  $20$ &  $50$ &  $100$&    $10$ &  $20$ &  $50$ &  $100$\\
   \hline\hline 
 \multicolumn{13}{||c|}{\textit{Wave}}\\\hline
 \multicolumn{1}{||l|}{BlockJS} &$35.82$ &$43.46$& $43.91$& $45.95$&$26.96$ &$27.34$& $28.22$& $28.45$&$19.02$ &$25.12$& $26.16$& $26.78$\\
 \hline  
 \multicolumn{1}{||l|}{BlockH}  &$34.08$ &$43.41$& $43.83$& $44.60$&$26.90$ &$27.34$& $28.19$& $28.41$&$19.21$ &$25.16$& $26.06$& $26.68$\\
 \hline 
 \multicolumn{1}{||l|}{TermJS}  &$33.30$ &$43.16$& $43.24$& $44.48$&$26.67$ &$27.25$& $28.14$& $28.36$&$18.32$ &$18.89$& $22.87$& $26.65$\\
 \hline  
 \multicolumn{1}{||l|}{TermH}   &$33.22$ &$40.16$& $39.32$& $44.27$&$26.77$ &$27.35$& $28.15$& $28.39$&$18.61$ &$19.07$& $18.95$& $26.68$\\
 \hline 
 \multicolumn{13}{||c|}{\textit{TimeShiftedSine}}\\\hline
 \multicolumn{1}{||l|}{BlockJS} &$38.69$ &$39.28$& $41.24$& $48.38$&$30.12$  &$38.85$& $39.26$& $40.19$&$22.58$ &$22.84$& $23.37$& $23.53$\\
 \hline  
 \multicolumn{1}{||l|}{BlockH}  &$38.66$ &$38.97$& $40.12$& $45.34$&$26.66$ &$38.22$& $39.22$& $39.08$&$22.58$ &$22.73$& $23.33$& $23.48$\\
 \hline 
 \multicolumn{1}{||l|}{TermJS}  &$38.41$ &$38.68$& $40.31$& $41.45$&$25.68$ &$29.24$& $36.79$& $36.68$&$22.52$ &$22.82$& $15.27$& $23.26$\\
 \hline  
 \multicolumn{1}{||l|}{TermH}   &$38.23$ &$38.44$& $39.07$& $41.46$&$26.98$&$28.10$& $35.18$& $36.98$&$22.49$ &$22.72$& $15.47$& $17.56$\\
 \hline 
 \multicolumn{13}{||c|}{\textit{Parabolas}}\\\hline
 \multicolumn{1}{||l|}{BlockJS} &$35.26$ &$37.04$& $40.64$& $44.74$&$22.03$ &$24.56$& $24.77$& $25.56$&$12.76$& $12.79$&$13.09$& $13.29$\\
 \hline  
 \multicolumn{1}{||l|}{BlockH}  &$34.12$ &$35.53$& $39.29$& $43.41$&$22.22$ &$24.54$& $24.47$& $25.00$ &$12.64$ &$12.76$& $12.78$& $12.77$\\
 \hline 
 \multicolumn{1}{||l|}{TermJS}  &$34.29$ &$35.18$& $39.38$& $41.37$&$21.70$ &$24.51$& $24.63$& $24.96$&$12.56$ &$12.76$& $12.79$& $12.79$\\
 \hline  
 \multicolumn{1}{||l|}{TermH}   &$33.27$ &$34.53$& $39.45$& $42.21$&$21.88$ &$23.98$& $24.21$& $24.86$&$12.70$ &$12.76$& $12.79$& $12.79$\\
 \hline\\
\end{tabular}
\caption{Comparison of average PSNR in decibels (dB) over 10 realizations of the noise for $d=0$, $d=1$ and $d=2$. From top to bottom BSNR$=40,25,10$ dB.}
 \label{tab:psnrcomp2}
\end{table}

We finally report a simulation study by quantitatively comparing BlockJS to the other thresholding estimators described above. For each test function, we generated $T = 4096$ equispaced samples on $[0,1]$ according to \eqref{sous} with varying number of channels ranging from $n=10$ to 100. 

\textsc{Table}~\ref{tab:psnrcomp2} summarizes the results. It shows in particular that BlockJS consistently outperforms the other methods in almost all cases in terms of PSNR. As expected and predicted by our theoretical findings, on the one hand, the performance gets better as the number of channels increases. On the other hand, it degrades with increasing noise level and/or $d$. Indeed, the derivatives estimation for BSNR$=10$ dB is rather difficult to estimate, especially for functions having highly irregular derivatives such as ``Parabolas'' (which has big jumps in its second derivative, see \textsc{Fig}.~\ref{fig:monopsnr}(d)).

\section{Conclusion and perspectives}
\label{sec:conclusion}
In this work, an adaptive wavelet block thresholding estimator was constructed to estimate one of the derivative of a function $f$ from the heteroscedastic multichannel deconvolution model. 
Under ordinary smooth assumption on $g_1,\ldots,g_n$, it was proved that it is nearly optimal in the minimax sense. The practical comparisons to state-of-the art methods have demonstrated the usefulness and the efficiency of adaptive block thresholding methods in estimating a function $f$ and its first derivatives in the functional deconvolution setting.

It would be interesting to consider the case where $g_v$ are unknown, which is the case in many practical situations.
Another interesting perspective would be to extend our results to a multidimensional setting. These aspects need further investigations that we leave for a future work. 

\section{Proofs}
\label{sec:proofs}
In the following proofs, $c$ and $C$ denote positive constants which can take different values for each mathematical term.

\subsection{Preparatory results}
In the three following results, we consider the framework of Theorem \ref{maison} and, for any  integer $j\ge j_*$ and $k\in \{1,\ldots,2^j-1\}$, we set $\alpha_{j,k} = \int_{0}^{1}f^{(d)}(t)\ophi_{j,k}(t)dt$ and $\beta_{j,k} = \int_{0}^{1}f^{(d)}(t)\opsi_{j,k}(t)dt$, the wavelet coefficients \eqref{coef} of $f^{(d)}$.

\begin{proposition}[Gaussian distribution on the wavelet coefficient estimators]\label{cooopp}
For any integer $j\ge j_*$ and $k\in \{0,\ldots,2^j-1\}$, we have
\[
 \widehat \alpha_{j,k}\sim \mathcal{N}\left(\alpha_{j,k},\epsilon^2 \frac{1}{\rho_n^2}\sum_{v=1}^n \frac{1}{(1+\sigma_v^2)^{2\delta}}\sum_{\ell\in\mathcal{D}_{j}}(2\pi \ell)^{2d} \frac{|\FT\left(\phi_{j,k}\right)(\ell)|^2}{|\FT(g_v)(\ell)|^2} \right)
\]
and
\[
 \widehat \beta_{j,k}\sim \mathcal{N}\left(\beta_{j,k},\epsilon^2 \frac{1}{\rho_n^2}\sum_{v=1}^n \frac{1}{(1+\sigma_v^2)^{2\delta}}\sum_{\ell\in\mathcal{C}_{j}}(2\pi \ell)^{2d} \frac{|\FT\left(\psi_{j,k}\right)(\ell)|^2}{|\FT(g_v)(\ell)|^2} \right).
\]
\end{proposition}

\begin{proof}
Let us prove the second point, the first one can be proved in a similar way. For any $\ell\in \mathbb{Z}$ and any $v\in\{1,\ldots,n\}$, $\FT\left( f\star g_v\right)(\ell)= \FT(f)(\ell)\FT(g_v)(\ell)$. 
Therefore, if we set
$$y_{\ell,v}=\int_{0}^{1}e^{- 2\pi i\ell t}dY_v(t), \qquad e_{\ell,v}=\int_{0}^{1}e^{- 2\pi i\ell t}dW_v(t),$$
It follows from \eqref{sous} that
\begin{equation}\label{dent}
 y_{\ell,v}=\FT(f)(\ell)\FT(g_v)(\ell)+\epsilon e_{\ell,v}.
\end{equation}

Note that, since $f$ is $1$-periodic, for any $u\in \{0,\ldots,d\}$, $f^{(u)}$ is $1$-periodic and $f^{(u)}(0)=f^{(u)}(1)$. By classical properties of the Fourier series, for any $\ell\in\mathbb{Z}$, we have $\FT(f^{(d)})(\ell)=(2\pi i\ell)^d\FT(f)(\ell)$. The Parseval theorem gives
\begin{align*}
\beta_{j,k}& = \int_{0}^{1}f^{(d)}(t)\opsi_{j,k}(t)dt=\sum_{\ell\in\mathcal{C}_{j}}\FT(f^{(d)})(\ell)\overline{\FT\left(\psi_{j,k}\right)}(\ell)\\
& = \sum_{\ell\in\mathcal{C}_{j}}(2\pi i\ell)^d\FT(f)(\ell)\overline{\FT\left(\psi_{j,k}\right)}(\ell).
\end{align*}
Using \eqref{dent}, we have
\begin{align*}
\widehat \beta_{j,k}& = \frac{1}{\rho_n}\sum_{v=1}^n \frac{1}{(1+\sigma_v^2)^{\delta}}\sum_{\ell\in\mathcal{C}_{j}}(2\pi i\ell)^d\frac{\overline{\FT\left(\psi_{j,k}\right)}(\ell)}{\FT(g_v)(\ell)}\FT(f)(\ell)\FT(g_v)(\ell)\\
& +
\epsilon \frac{1}{\rho_n}\sum_{v=1}^n \frac{1}{(1+\sigma_v^2)^{\delta}}\sum_{\ell\in\mathcal{C}_{j}}(2\pi i\ell)^d\frac{\overline{\FT\left(\psi_{j,k}\right)}(\ell)}{\FT(g_v)(\ell)}e_{\ell,v}\\
& =
\sum_{\ell\in\mathcal{C}_{j}}(2\pi i\ell)^d \FT(f)(\ell) \overline{\FT\left(\psi_{j,k}\right)}(\ell)  \\
& +
\epsilon \frac{1}{\rho_n}\sum_{v=1}^n \frac{1}{(1+\sigma_v^2)^{\delta}}\sum_{\ell\in\mathcal{C}_{j}}(2\pi i\ell)^d \frac{\overline{\FT\left(\psi_{j,k}\right)}(\ell)}{\FT(g_v)(\ell)}e_{\ell,v}\\
& =
\beta_{j,k}  +\epsilon \frac{1}{\rho_n}\sum_{v=1}^n \frac{1}{(1+\sigma_v^2)^{\delta}}\sum_{\ell\in\mathcal{C}_{j}}(2\pi i\ell)^d \frac{\overline{\FT\left(\psi_{j,k}\right)}(\ell)}{\FT(g_v)(\ell)}e_{\ell,v}.
\end{align*}

Since $\{e^{- 2\pi i\ell .}\}_{\ell\in\mathbb{Z}}$ is an orthonormal basis of $\Lper([0,1])$ and $W_1(t),\ldots,W_n(t)$ are i.i.d. standard Brownian motions, $\left(\int_{0}^{1}e^{- 2\pi i\ell t}dW_v(t)\right)_{(\ell,v)\in \mathbb{Z}\times \{1,\ldots,n\}}$ is a sequence of i.i.d. random variables with the common distribution $\mathcal{N}(0,1)$. Therefore $$\widehat \beta_{j,k}\sim \mathcal{N}\left(\beta_{j,k},\epsilon^2 \frac{1}{\rho_n^2}\sum_{v=1}^n \frac{1}{(1+\sigma_v^2)^{2\delta}}\sum_{\ell\in\mathcal{C}_{j}}(2\pi \ell)^{2d} \frac{|\FT\left(\psi_{j,k}\right)(\ell)|^2}{|\FT(g_v)(\ell)|^2} \right).$$
Proposition \ref{cooopp} is proved.
\end{proof}

\begin{proposition}[Moment inequalities]\label{copp} \ 
\begin{itemize}
\item There exists a constant $C>0$ such that, for any integer $j\ge j_*$ and $k\in \{0,\ldots,2^j-1\}$,  $$\mathbb{E}\left( |\widehat \alpha_{j_1,k}-\alpha_{j_1,k}|^2\right)\le C \epsilon^2 2^{2 (\delta+d) j_1}\rho_n^{-1},$$
\item There exists a constant $C>0$ such that, for any integer $j\ge j_*$ and $k\in \{0,\ldots,2^j-1\}$, $$\mathbb{E}\left( |\widehat \beta_{j,k}-\beta_{j,k}|^4\right)\le C \epsilon^4 2^{4 (\delta+d) j}\rho_n^{-2}.$$
\end{itemize}
\end{proposition}

\begin{proof} 
Let us prove the second point, the first one can be proved in a similar way. Let us recall that, by Proposition \ref{cooopp}, for any $j\in \{j_1,\ldots,j_2\}$ and any $k\in \{0,\ldots,2^{j}-1\}$, we have 
\begin{equation}\label{distr}
\widehat \beta_{j,k}-\beta_{j,k}\sim\mathcal{N}\left(0,\rho_n^{-2}\sigma_{j,k}^2\right),
\end{equation}
 where
\begin{equation}\label{tabb}
\sigma_{j,k}^2=\epsilon^2\sum_{v=1}^n \frac{1}{(1+\sigma_v^2)^{2\delta}}\sum_{\ell\in\mathbb{Z}}(2\pi \ell)^{2d} \frac{|\FT\left(\psi_{j,k}\right)(\ell)|^2}{|\FT(g_v)(\ell)|^2}.
\end{equation}
Due to \eqref{cond} and \eqref{cal}, for any $v\in \{1,\ldots,n\}$, we have
\begin{align}\label{dol}
\sup_{\ell \in \mathcal{C}_j}\left(\frac{(2\pi \ell)^{2d}}{|\FT(g_v)(\ell)|^2}\right)
& \le  
C \sup_{\ell \in \mathcal{C}_j}\left((2\pi \ell)^{2d}\left(1+\sigma_v^2\ell^2\right)^{\delta} \right) \nonumber \\
& \le  
C (1+\sigma_v^2)^{\delta}\sup_{\ell \in \mathcal{C}_j}\left((2\pi \ell)^{2d}\left(1+\ell^2\right)^{\delta} \right) \nonumber \\
& \le   
C (1+\sigma_v^2)^{\delta} 2^{2(\delta+d) j}.
\end{align}
It follows from \eqref{dol} and the Parseval identity that
\begin{align}\label{sigg}
\sigma_{j,k}^2& \le  \epsilon^2\sum_{v=1}^n \frac{1}{(1+\sigma_v^2)^{2\delta}}\sup_{\ell \in \mathcal{C}_j}\left(\frac{(2\pi \ell)^{2d}}{|\FT(g_v)(\ell)|^2}\right)\sum_{\ell\in\mathcal{C}_j}|\FT\left(\psi_{j,k}\right)(\ell)|^2 \nonumber \\
& \le  
C \epsilon^22^{2(\delta+d) j} \left( \sum_{v=1}^n \frac{1}{(1+\sigma_v^2)^{\delta}}\right) \left(\sum_{\ell\in\mathcal{C}_j}|\FT\left(\psi_{j,k}\right)(\ell)|^2\right) \nonumber \\
& =  
C \epsilon^22^{2(\delta+d) j}\rho_n \int_{0}^1|\psi_{j,k}(t)|^2dt=C \epsilon^2 \rho_n2^{2(\delta+d) j}.
\end{align}
Putting \eqref{distr}, \eqref{tabb} and \eqref{sigg} together, we obtain
\[
 \mathbb{E}\left( |\widehat \beta_{j,k}-\beta_{j,k}|^4\right)\le C (\epsilon^22^{2(\delta+d) j} \rho_n \rho_n^{-2})^2 =  C \epsilon^4 2^{4 (\delta+d) j}\rho_n^{-2}.
\]
Proposition \ref{copp} is proved.
\end{proof}

\begin{proposition}[Concentration inequality]\label{cop}
There exists a constant $\lambda>0$ such that,
for any $j\in \{j_1,\ldots,j_2\}$, any $K\in \mathcal{A}_j$ and $n$ large enough,
\begin{equation*}
\mathbb{P} \left( \left(\sum_{k\in B_{j,K}}|\widehat \beta_{j,k}-\beta_{j,k}|^2\right)^{1/2}  \ge \lambda  2^{(\delta+d)j}(\log \rho_n/\rho_n)^{1/2} \right)  \le \rho_n^{-2}.
\end{equation*}
\end{proposition}

\begin{proof}
We need the Tsirelson inequality stated in Lemma\ref{cirel} below.
\begin{lemma}[\cite{cirelson}]\label{cirel} Let $(\vartheta_t)_{t\in D}$ be a centered Gaussian process.
If 
\[
 \mathbb{E}\left(\sup_{t\in D}\vartheta_t\right)\le N, \qquad \sup_{t\in D} \mathbb{V}\left(\vartheta_t\right)\le V
\]
then, for any $x>0$, we have
\begin{equation*}
\mathbb{P}\left(\sup_{t\in {D}}\vartheta_t\ge x+N\right)\le \exp\left(-\frac{x^2}{2V}\right).
\end{equation*}
\end{lemma}
For the sake of notational clarity, let 
\[
 V_{j,k}=\widehat\beta_{j,k}-\beta_{j,k} .
\]
Recall that, by Proposition \ref{cooopp}, we have $V_{j,k}\sim\mathcal{N}\left(0,\rho_n^{-2}\sigma_{j,k}^2\right)$, where $\sigma_{j,k}^2$ is given in \eqref{tabb}. 
Let $\mathbb{B}(1)$ the unit $2$-norm ball in $\mathbb{C}^{\card(B_{j,K})}$, i.e. $\mathbb{B}(1)=\lbrace a \in \mathbb{C}^{\card(B_{j,K})}; \ \sum_{k\in B_{j,K}}|a_k|^2\le 1 \rbrace$. 
For any $a\in \mathbb{B}(1)$, let ${Z}(a)$ be the centered Gaussian process defined by 
\begin{align*}
 Z(a) & = \sum_{k\in B_{j,K}}a_{k} V_{j,k}\\
&= \epsilon \frac{1}{\rho_n}\sum_{v=1}^n \frac{1}{(1+\sigma_v^2)^{\delta}}\sum_{\ell\in\mathcal{C}_j}(2\pi i\ell)^{d}\frac{e_{\ell,v}}{\FT(g_v)(\ell)} \sum_{k\in B_{j,K}}a_{k}\overline{\FT\left(\psi_{j,k}\right)}(\ell).
\end{align*}

By a simple Legendre-Fenchel conjugacy argument, we have 
\[
\sup_{a\in \mathbb{B}(1)}{Z}(a) =\left(\sum_{k\in B_{j,K}}|V_{j,k}|^2\right)^{1/2}=\left(\sum_{k\in B_{j,K}}|\widehat \beta_{j,k}-\beta_{j,k}|^2\right)^{1/2}.
\]
Now, let us determine the values of $N$ and $V$ which appeared in the Tsirelson inequality.

\paragraph*{Value of $N$.} 
Using the Jensen inequality and \eqref{sigg}, we obtain
\begin{align*}
\mathbb{E}\left( \sup_{a\in \mathbb{B}(1)} {{Z}}(a)\right) & = \mathbb{E}\left( \left(\sum_{k\in B_{j,K}}|V_{j,k}|^2\right)^{1/2}\right) \le \left(\sum_{k\in B_{j,K}}\mathbb{E} \left( |V_{j,k}|^2\right)\right)^{1/2}\\
& \le C \left(\rho_n^{-2}\sum_{k\in B_{j,K}}\sigma_{j,k}^2\right)^{1/2}\le C \left(\rho_n^{-2}\epsilon^2 \rho_n2^{2(\delta+d) j}\card(B_{j,K})\right)^{1/2} \\
& \leq C\epsilon 2^{(\delta+d) j} \left(\log \rho_n /\rho_n\right)^{1/2} ~.
\end{align*}
Hence $N= C\epsilon 2^{(\delta+d) j}(\log \rho_n/\rho_n)^{1/2}$.

\paragraph*{Value of $V$.}
Note that, for any $(\ell, \ell') \in \mathbb{Z}^2$ and any $(v,v')\in\{1,\ldots,n\}^2$, 
\[
\mathbb{E}\left(e_{\ell,n}\overline{e}_{\ell',v'}\right)=
\begin{cases} 1 & {\rm if} \ \ell=\ell' \ {\rm and} \  v=v',\\
0 & {\rm otherwise}.
\end{cases}
\]
It then follows that 
\begin{align}\label{collier}
&\sup_{a\in  {\mathbb{B}(1)} } \mathbb{V} ( {Z}(a) )  =  \sup_{a\in \mathbb{B}(1)} \mathbb{E} \left( \left|  \sum_{k\in B_{j,K}}a_kV_{j,k}\right|^2\right) \nonumber \\
&=\sup_{a\in \mathbb{B}(1)} \mathbb{E} \left(  \sum_{k\in B_{j,K}}\sum_{k'\in B_{j,K}} {a_k}\overline{a}_{k'}V_{j,k}\overline{V}_{j,k'}\right)\nonumber  \\
&=\epsilon^2 \rho_n^{-2}\sup_{a\in \mathbb{B}(1)} \sum_{k\in B_{j,K}}\sum_{k'\in B_{j,K}} {a_{k}} \overline{a}_{k'} \sum_{\ell\in\mathcal{C}_j}\sum_{\ell'\in \mathcal{C}_j}\sum_{v'=1}^n\sum_{v=1}^n\frac{1}{(1+\sigma_v^2)^{\delta}}\frac{1}{(1+\sigma_{v'}^2)^{\delta}}\times \nonumber \\
&~~~~\frac{(2\pi i \ell)^d}{\FT(g_v)(\ell)}\overline{\FT(\psi_{j,k} )}(\ell)\frac{\overline{(2\pi i \ell ')^d}}{\overline{\FT(g_{v'})}(\ell ')}{\FT(\psi_{j,k'})(\ell ')} \mathbb{E}\left(e_{\ell,v}\overline{e}_{\ell',v'}\right) \nonumber \\
&=\epsilon^2 \rho_n^{-2}\sup_{a\in \mathbb{B}(1)}  \sum_{k\in B_{j,K}}\sum_{k'\in B_{j,K}} {a_{k}}\overline{a}_{k'}  \sum_{\ell\in\mathcal{C}_j} \sum_{v=1}^n\frac{1}{(1+\sigma_v^2)^{2\delta}}\frac{(2\pi  \ell)^{2d}}{|\FT(g_v)(\ell)|^2} \overline{\FT\left(\psi_{j,k}\right)}(\ell){\FT(\psi_{j,k'})(\ell)} \nonumber \\
&= \epsilon^2 \rho_n^{-2}\sup_{a\in \mathbb{B}(1)} \sum_{\ell\in\mathcal{C}_j}\sum_{v=1}^n\frac{1}{(1+\sigma_v^2)^{2\delta}} \frac{(2\pi  \ell)^{2d}}{|\FT(g_v)(\ell)|^2} \left|\sum_{k\in B_{j,K}} \overline{a}_{k}\FT\left(\psi_{j,k}\right)(\ell)\right|^2.
\end{align}

For any $a\in \mathbb{B}(1)$, the Parseval identity and the fact that $\{\psi_{j,k}\}_{k=0,\ldots,2^{j}-1}$ are orthonormal yields
\begin{align}\label{rock}
\sum_{\ell\in\mathcal{C}_j}\left|\sum_{k\in B_{j,K}} \overline{a}_{k}\FT(\psi_{j,k} )(\ell)\right|^2 
& = 
\sum_{\ell\in\mathcal{C}_j} \left|\FT\left(\sum_{k\in B_{j,K}} \overline{a}_{k}\psi_{j,k} \right)(\ell)\right|^2 \nonumber \\
& =   
\int_{0}^{1} \left|\sum_{k\in B_{j,K}} \overline{a}_{k}\psi_{j,k}(t)\right|^2dt\nonumber \\
& =  
\sum_{k\in B_{j,K}} |a_k|^2\le 1.
\end{align}
Piecing \eqref{dol}, \eqref{collier} and \eqref{rock} together, we get
\begin{align*}
\sup_{a\in  {\mathbb{B}(1)} } \mathbb{V} ( {Z}(a) ) & \le   C\epsilon^2 \rho_n^{-1}2^{2(\delta+d) j}\sup_{a\in \mathbb{B}(1)} \sum_{\ell\in\mathcal{C}_j} \left| \sum_{k\in B_{j,K}} \overline{a}_{k}\FT\left(\psi_{j,k}\right)(\ell)\right|^2\\
& \le C\epsilon^2 \rho_n^{-1}2^{2(\delta+d) j}.
\end{align*}
Hence it is sufficient to take $V=C\epsilon^2 \rho_n^{-1}2^{2(\delta+d) j}$.

Taking $\lambda$ large enough and $x=2^{-1}{\lambda} \epsilon  2^{(\delta+d) j}(\log \rho_n/\rho_n)^{1/2}$, the Tsirelson inequality (see Lemma \ref{cirel}) yields
\begin{align*}
\lefteqn{\mathbb{P} \left( \left(\sum_{k\in B_{j,K}} |V_{j,k}|^2\right)^{1/2}  \ge {\lambda} \epsilon 2^{(\delta+d) j}(\log \rho_n/\rho_n)^{1/2}\right) } &  \nonumber \\
& \le
\mathbb{P} \left( \left(\sum_{k\in B_{j,K}}|V_{j,k}|^2\right)^{1/2}  \ge 2^{-1}{\lambda} \epsilon 2^{(\delta+d) j}(\log \rho_n/\rho_n)^{1/2}+ N\right) \nonumber \\
& =
\mathbb{P}\left( \sup_{a\in \mathbb{B}(1)} {{Z}}(a) \ge x+N\right)\le \exp\left(-{x^2}/{(2V)}\right)\le \exp\left( -C\lambda^2 \log \rho_n\right)\nonumber \\
& \le
\rho_n^{-2}.
\end{align*}
Proposition \ref{cop} is proved.
\end{proof}

\subsection{Proof of Theorem~\ref{maison}}
\begin{proof}
Plugging Propositions \ref{copp} and \ref{cop} into \cite[Theorem 3.1]{cfs}, we end the proof of Theorem \ref{maison}.
\end{proof}

\subsection{Proof of Theorem~\ref{maison2}} 
\begin{proof}
Let us now present a consequence of the Fano lemma.

\begin{lemma}\label{fanost} 
Let $m\in \mathbb{N}^*$ and $A$ be a $\sigma$-algebra on the space $\Omega$. For any $i\in \{0,\ldots,m\}$, let $A_i\in A$ such that, for any $(i,j)\in \{0,\ldots,m\}^2$ with $i\ne j$, $$A_i \cap A_j=\varnothing.$$ Let $(\mathbb{P}_i)_{i\in \{0,\ldots,m\}}$ be $m+1$ probability measures on $(\Omega, A)$. Then
\[
 \sup_{i\in \{0,\ldots,m\}}\mathbb{P}_i\left(A_i^c\right)\ge \min\left( 2^{-1}, \exp(-{3}e^{-1})\sqrt{m}\exp(- \chi_m)\right),
\]
where
\begin{equation*}
\chi_m=\inf_{v\in \{0,\ldots,m\}}\frac{1}{m}\sum_{\underset{k\ne v}{k\in \{0,\ldots,m\}}}K(\mathbb{P}_k,\mathbb{P}_v),
\end{equation*}
$A^c$ denotes the complement of $A$ in $\Omega$ and $K$ is the Kullbak-Leibler divergence defined by
\begin{equation*}
K(\mathbb{P},\mathbb{Q})=
\begin{cases} \int \ln\left( \frac{d \mathbb{P}}{d \mathbb{Q}}\right)d \mathbb{P} & \rm{if} \ \ \mathbb{P} \ll \mathbb{Q}, \\
 \infty & \rm{otherwise}.
\end{cases}
\end{equation*}
\end{lemma}
The proof of Lemma \ref{fanost} can be found in \cite[Lemma 3.3]{kertem}. For further details and applications of the Fano lemma, see \cite{tsybakov}.

In what follows, we distinguish $\psi$ and $\psi^{\mathrm{per}}$ (see Section \ref{perio}).  For any integrable function $h$ on $\mathbb{R}$, we set
\[
 {\FT}_*(h)(\ell)=\int_{-\infty}^{\infty}h(t)e^{-2 i\pi \ell t}dt, \qquad \ell\in \mathbb{Z}.
\]

Consider the Besov balls $\Bspr$ (see \eqref{besovv}). Let $j_0$ be an integer suitably chosen below. For any $\varepsilon=(\varepsilon_k)_{k\in\{0,\ldots,2^{j_0}-1\}}\in \{0,1\}^{2^{j_0}} $ and $d\in \mathbb{N}^*$, set
\begin{equation*}
h_{\varepsilon}(t) = M_* 2^{-j_0(s+1/2)} \sum_{k=0}^{2^{j_0}-1}\varepsilon_k \frac{1}{(d-1) !}\sum_{l\in\mathbb{Z}}\int_{-\infty}^{t+l}(t+l-y)^{d-1}\psi_{j_0,k}(y)dy, ~ t \in \lbrack 0,1 \rbrack,
\end{equation*}
and, if $d=0$, set $h_{\varepsilon}(t)=M_* 2^{-j_0(s+1/2)} \sum_{k=0}^{2^{j_0}-1}\varepsilon_k\psi^{\mathrm{per}}_{j_0,k}(t), t \in [0,1]$. Notice that, owing to \eqref{ratratrat}, $h_{\varepsilon}$ exists. Moreover, it is $1$-periodic. Using the Cauchy formula for repeated integration, we have
\[
h^{(d)}_{\varepsilon}(t)=M_*2^{-j_0(s+1/2)} \sum_{k=0}^{2^{j_0}-1}\varepsilon_k\psi^{\mathrm{per}}_{j_0,k}(t), \qquad t \in \lbrack 0,1 \rbrack.
\]
So, for any $j\ge j_*$ and any $k\in \{0,\ldots,2^{j}-1\}$, the (mother) wavelet coefficient of $h^{(d)}_{\varepsilon}$ is
\begin{align*}
\beta_{j,k}=\int_{0}^{1}h^{(d)}_{\varepsilon}(t) \opsi^{\mathrm{per}}_{j,k}(t)dt=
\begin{cases}
M_*\varepsilon_k 2^{-j_0(s+1/2)}, & \ {\rm if} \ j=j_0,\\
0, & \ {\rm otherwise.}
\end{cases}
\end{align*}
Therefore $h_{\varepsilon}^{(d)}\in \Bspr$. The Varshamov-Gilbert theorem (see \cite[Lemma 2.7]{tsybakov}) asserts that there exist a set 
$E_{j_0}=\left\lbrace\varepsilon^{(0)},\ldots,\varepsilon^{(T_{j_0})}\right\rbrace$ and two constants, $c\in\rbrack0,1\lbrack$ and $\alpha\in\rbrack0,1\lbrack$, 
such that, for any $u\in\{0,\ldots,T_{j_0}\}$, $\varepsilon^{(u)}=(\varepsilon_{k}^{(u)})_{k\in\{0,\ldots,2^{j_0}-1\}}\in\{0,1\}^{2^{j_0}}$ and any $(u,v)\in\{0,\ldots,T_{j_0}\}^2$ with $ u<v$, the following inequalities hold:
\begin{equation*}
\sum_{k=0}^{2^{j_0}-1}|\varepsilon_k^{(u)}-\varepsilon^{(v)}_{k}|\ge c{2^{j_0}}, \qquad T_{j_0}\ge e^{\alpha 2^{j_0}}.
\end{equation*}
Considering such a set $E_{j_0}$, for any $(u,v)\in \{0,\ldots,T_{j_0}\}^2$ with $u\ne v$, we have by orthonormality of the collection $\{\psi_{j_0,k}\}_{k=0,\ldots,2^{j_0}-1}$
\begin{align*}
\lefteqn{\left(\int_{0}^{1} \left(h^{(d)}_{\varepsilon^{(u)}}(t)-h^{(d)}_{\varepsilon^{(v)}}(t)\right)^2dt\right)^{1/2}} & \\
&  =  M_* c 2^{-j_0(s+1/2)} \left(\sum_{k=0}^{2^{j_0}-1}\left|\varepsilon_k^{(u)}-\varepsilon_k^{(v)}\right|^2\right)^{1/2}\\
&  =  M_* c 2^{-j_0(s+1/2)} \left(\sum_{k=0}^{2^{j_0}-1}\left|\varepsilon_k^{(u)}-\varepsilon_k^{(v)}\right|\right)^{1/2}\\
& \ge 2 \delta_{j_0},
\end{align*}
where
\[
\delta_{j_0}= M_* c^{1/2}2^{j_0/2}2^{-j_0(s+1/2)}=M_* c^{1/2}2^{-j_0s}.
\]
Using the Markov inequality, for any estimator $\widetilde{f^{(d)}}$ of $f^{(d)}$, we have
\begin{equation*}
 \delta_{j_0}^{-2}\sup_{f^{(d)}\in \Bspr}\mathbb{E} \left(\int_{0}^{1} \left(\widetilde{f^{(d)}}(t)-f^{(d)}(t)\right)^2dt\right)\ge \sup_{u\in \{0,\ldots,T_{j_0}\}}\mathbb{P}_{h_{\varepsilon^{(u)}}}(A_u^{c}) = p,
\end{equation*}
where 
\[
A_u=\left \lbrace \left(\int_{0}^{1} \left(\widetilde{f^{(d)}}(t)-h_{\varepsilon^{(u)}}^{(d)}(t)\right)^2dt\right)^{1/2} <\delta_{j_0}\right \rbrace
\] 
and  $\mathbb{P}_f$ is the distribution of model \eqref{sous}. Notice that, for any $(u,v)\in \{0,\ldots,T_{j_0}\}^2$ with $u\not= v$, $A_u\cap A_{v}=\varnothing$. Lemma \ref{fanost} applied to the probability measures $\left(\mathbb{P}_{h_{\varepsilon^{(u)}}}\right)_{u\in \{0,\ldots,T_{j_0}\}}$ gives 
\begin{equation}\label{eternity}
p\ge \min\left( 2^{-1}, \exp(-{3}{e}^{-1})\sqrt{T_{j_0}}\exp(-\chi_{T_{j_0}})\right),
\end{equation}
where
\[
\chi_{T_{j_0}}=\inf_{v\in \{0,\ldots,T_{j_0}\}}\frac{1}{T_{j_0}}\sum_{\underset{u\ne
v}{u\in \{0,\ldots,T_{j_0}\}}}K\left(\mathbb{P}_{h_{\varepsilon^{(u)}}},\mathbb{P}_{h_{\varepsilon^{(v)}}}\right).
\]
Let's now  bound $\chi_{T_{j_0}}$. For any functions $f_1$ and $f_2$ in $\Lper([0,1])$, we have
\begin{align*}
K\left(\mathbb{P}_{f_1},\mathbb{P}_{f_2}\right)& = \frac{1}{2\epsilon^2}\sum_{v=1}^n\int_{0}^{1}  \left( (f_1\star g_v)(t)-(f_2\star g_v)(t)\right)^2dt\\
& = \frac{1}{2\epsilon^2}\sum_{v=1}^n\int_{0}^{1}  \left(( (f_1-f_2)\star g_v)(t)\right)^2dt.
\end{align*}
The Parseval identity yields
\begin{align*}
K\left(\mathbb{P}_{f_1},\mathbb{P}_{f_2}\right)& = \frac{1}{2\epsilon^2}\sum_{v=1}^n\sum_{\ell \in \mathbb{Z}} | \FT( (f_1-f_2)\star g_v)(\ell)|^2\\
& = \frac{1}{2\epsilon^2}\sum_{v=1}^n \sum_{\ell \in \mathbb{Z}}\left|\FT(f_1-f_2)(\ell)\right|^2 |\FT(g_v)(\ell)|^2.
\end{align*}
So, for any $(u,v)\in \{0,\ldots,T_{j_0}\}^2$ with $u\ne v$, we have
\begin{equation}\label{vanish}
K\left(\mathbb{P}_{h_{\varepsilon^{(u)}}}, \mathbb{P}_{h_{\varepsilon^{(v)}}}\right)=\frac{1}{2\epsilon^2}\sum_{v=1}^n\sum_{\ell \in \mathbb{Z}} \left|\FT\left(h_{\varepsilon^{(u)}}-h_{\varepsilon^{(v)}}\right)(\ell)\right|^2 |\FT(g_v)(\ell)|^2.
\end{equation}
By definition, for any $(u,v)\in \{0,\ldots,T_{j_0}\}^2$ with $u\ne v$ and $\ell \in \mathbb{Z}$, we have
\begin{align}\label{ret}
\lefteqn{\FT\left(h_{\varepsilon^{(u)}}-h_{\varepsilon^{(v)}}\right)(\ell)} & \nonumber \\
& = M_* 2^{-j_0(s+1/2)} \sum_{k=0}^{2^{j_0}-1}\left(\varepsilon^{(u)}_k-\varepsilon^{(v)}_k\right) \times \nonumber \\
& \frac{1}{(d-1) !}\sum_{l\in\mathbb{Z}}\FT\left(\int_{-\infty}^{\cdot+l}(\cdot+l-y)^{d-1}\psi_{j_0,k}(y)dy\right)(\ell)\nonumber \\  
& = M_* 2^{-j_0(s+1/2)} \sum_{k=0}^{2^{j_0}-1}\left(\varepsilon^{(u)}_k-\varepsilon^{(v)}_k\right) \times \nonumber \\
& \frac{1}{(d-1) !}{\FT}_*\left(\int_{\infty}^\cdot(\cdot-y)^{d-1}\psi_{j_0,k}(y)dy\right)(\ell). 
\end{align}
Let, for any $k\in \{0,\ldots,2^{j_0}-1\}$,
\[
\theta_k(t)=\frac{1}{(d-1)!}\int_{-\infty}^t(t-y)^{d-1}\psi_{j_0,k}(y)dy, \qquad t \in [0,1].
\] 
Then, for any $u\in \{0,\ldots,d\}$, thanks to \eqref{reg2}, $\lim_{|t|\rightarrow \infty}\theta_k^{(u)}(t)=0$. Consequently, for any $\ell \in \mathbb{Z}$,
\[
(2\pi i \ell)^d{\FT}_*(\theta_k)(\ell)={\FT}_*\left(\theta_k^{(d)}\right)(\ell)={\FT}(\psi^{\mathrm{per}}_{j_0,k})(\ell).
\]
So, for any $\ell \in \mathcal{C}_{j_0}$ (excluding $0$), \eqref{ret} implies that 
\begin{align}\label{ret2}
\lefteqn{\FT(h_{\varepsilon^{(u)}}-h_{\varepsilon^{(v)}})(\ell)} &  \nonumber \\
& = M_* 2^{-j_0(s+1/2)} \sum_{k=0}^{2^{j_0}-1}\left(\varepsilon^{(u)}_k-\varepsilon^{(v)}_k\right) \frac{1}{(2\pi i\ell)^d}\FT\left(\psi^{\mathrm{per}}_{j_0,k}\right)(\ell),
\end{align}
which entails in particular that $\supp\left(\FT(h_{\varepsilon^{(u)}}-h_{\varepsilon^{(v)}})\right) = \mathcal{C}_{j_0}$. This in conjunction with equalities \eqref{vanish} and \eqref{ret2} imply that
\begin{align}\label{debut}
&K\left(\mathbb{P}_{h_{\varepsilon^{(u)}}}, \mathbb{P}_{h_{\varepsilon^{(v)}}}\right)\nonumber\\
&=\frac{M_*^2}{2\epsilon^2}2^{-2j_0(s+1/2)}\sum_{v=1}^n\sum_{\ell \in \mathcal{C}_{j_0}}  \left|\sum_{k=0}^{2^{j_0}-1}\left(\varepsilon^{(u)}_k-\varepsilon^{(v)}_k\right) \FT\left(\psi^{\mathrm{per}}_{j_0,k}\right)(\ell)\right|^2\frac{1}{(2\pi \ell)^{2d}} |\FT(g_v)(\ell)|^2.
\end{align}

By assumptions \eqref{cond} and \eqref{cal}, for any $v\in \{1,\ldots,n\}$,
\begin{align}\label{horiz}
\sup_{\ell \in \mathcal{C}_{j_0}}\left(\frac{1}{(2\pi \ell)^{2d}} |\FT(g_v)(\ell)|^2\right)
& \le  C\sup_{\ell \in \mathcal{C}_{j_0}}\left(\frac{1}{(2\pi \ell)^{2d}}\left(1+\sigma_v^2\ell^2\right)^{-\delta}\right) \nonumber \\
& \le  C\sigma_v^{-2\delta}\sup_{\ell \in \mathcal{C}_{j_0}}\left(\ell^{-2(\delta+d)}\right) \le C\sigma_v^{-2\delta} 2^{-2j_0(\delta+d)}.
\end{align}
Moreover, the Parseval identity and orthonormality of the family $\left\{\psi^{\mathrm{per}}_{j_0,k}\right\}_{k=0,\ldots,2^{j_0}-1}$ imply that
\begin{align}\label{event}
\lefteqn{\sum_{\ell \in \mathcal{C}_{j_0}}  \left|\sum_{k=0}^{2^{j_0}-1}\left(\varepsilon^{(u)}_k-\varepsilon^{(v)}_k\right) \FT\left(\psi^{\mathrm{per}}_{j_0,k}\right)(\ell)\right|^2}& \nonumber \\ 
& = \int_{0}^{1}  \left|\sum_{k=0}^{2^{j_0}-1}\left(\varepsilon^{(u)}_k-\varepsilon^{(v)}_k\right) \psi^{\mathrm{per}}_{j_0,k}(t)\right|^2dt= \sum_{k=0}^{2^{j_0}-1}\left(\varepsilon^{(u)}_k-\varepsilon^{(v)}_k\right)^2\le 2^{j_0}. 
\end{align}
It follows from \eqref{debut}, \eqref{horiz} and \eqref{event} that
\begin{equation*}
K\left(\mathbb{P}_{h_{\varepsilon^{(u)}}},\mathbb{P}_{h_{\varepsilon^{(v)}}}\right)\le C2^{-2j_0(s+1/2)}2^{-2j_0(\delta+d)}2^{j_0}\sum_{v=1}^n\sigma_v^{-2\delta}=C\rho^*_{n}2^{-2j_0(s+1/2+\delta+d)}2^{j_0}.
\end{equation*}

Hence
\begin{align}\label{shin}
\chi_{T_{j_0}}& = \inf_{v\in \{0,\ldots,T_{j_0}\}}\frac{1}{T_{j_0}}\sum_{\underset{u\ne
v}{u\in \{0,\ldots,T_{j_0}\}}}K\left(\mathbb{P}_{h_{\varepsilon^{(u)}}},\mathbb{P}_{h_{\varepsilon^{(v)}}}\right)\nonumber \\
& \le C\rho^*_{n}2^{-2j_0(s+1/2+\delta+d)}2^{j_0}.
\end{align}
Putting \eqref{eternity} and \eqref{shin} together and choosing $j_0$ such that $$2^{-j_0(s+1/2+\delta+d)}= c_0(\rho^*_{n})^{-1/2},$$
where $c_0$ denotes a well chosen constant, for any estimator $\widetilde{f^{(d)}}$ of $f^{(d)}$, we have
\begin{align*}
\delta_{j_0}^{-2}\sup_{f^{(d)}\in \Bspr}\mathbb{E}\left(\int_{0}^{1}\left(\widetilde{f^{(d)}}(t)-f^{(d)}(t)\right)^2dt\right)& \ge c\exp\left( (\alpha /2) 2^{j_0}-Cc_0^2 2^{j_0}\right)\\
& \ge c,
\end{align*}
where $$\delta_{j_0}=c 2^{-j_0s}=c(\rho^*_{n})^{-s/(2s+2\delta+2d+1)}.$$
This complete the proof of Theorem \ref{maison2}.
\end{proof}

\paragraph*{Acknowledgements} 
The authors are very grateful to the Associate Editor and the anonymous referees for their careful reading and valuable comments that have led to an improvement of the presentation. This work is supported by ANR grant NatImages, ANR-08-EMER-009.

\bibliographystyle{plain}
\bibliography{dec-multi}

\end{document}